%% file: eh.tex
\documentclass[12pt]{article}

\title{Symplectic capacities from positive $S^1$-equivariant symplectic homology}
\author{Jean Gutt\footnote{Partially supported by a Belgian American Educational Foundation fellowship and by the SFB/TRR 191 ``Symplectic structures in Geometry, Algebra and Dynamics''}  and Michael Hutchings\footnote{Partially supported by NSF grant DMS-1406312 and a Simons Fellowship.}}
\date{}

\usepackage{amssymb}
\usepackage{latexsym}
\usepackage{amsmath}
\usepackage{amsthm}
\usepackage{amscd}
\usepackage[dvips]{graphics}
\usepackage{overpic}
\usepackage[mathscr]{euscript}
\usepackage{color}
\usepackage{xypic}
\usepackage{mathabx}
\xyoption{curve}

\newcommand{\mc}[1]{{\mathcal #1}}

\numberwithin{equation}{section}

\newtheorem{theorem}{Theorem}[section]
\newtheorem{proposition}[theorem]{Proposition}
\newtheorem{corollary}[theorem]{Corollary}
\newtheorem{lemma}[theorem]{Lemma}
\newtheorem{lemma-definition}[theorem]{Lemma-Definition}

\newtheorem{conjecture}[theorem]{Conjecture}

\theoremstyle{definition}
\newtheorem{definition}[theorem]{Definition}
\newtheorem{remark}[theorem]{Remark}

\newtheorem{example}[theorem]{Example}

\newcommand{\floor}[1]{\left\lfloor #1 \right\rfloor}

\newcommand{\eqdef}{\;{:=}\;}

\newcommand{\C}{{\mathbb C}}
\newcommand{\Q}{{\mathbb Q}}
\newcommand{\R}{{\mathbb R}}
\newcommand{\N}{{\mathbb N}}
\newcommand{\Z}{{\mathbb Z}}

\newcommand{\op}{\operatorname}

\newcommand{\M}{\mc{M}}

\newcommand{\End}{\op{End}}

\newcommand{\tensor}{\otimes}

\newcommand{\CZ}{\op{CZ}}

\newcommand{\Per}{\mathcal{P}}
\newcommand{\Mod}{\mathcal{M}}

\newcommand{\Hs}{\mathcal{H}}
\newcommand{\Spec}{\operatorname{Spec}}
\newcommand{\m}{\operatorname{min}}
\newcommand{\Max}{\operatorname{Max}}
\newcommand{\Hstd}{\mathcal{H}_{\textrm{adm}}}

\newcommand{\bpm}{\begin{pmatrix}}
\newcommand{\epm}{\end{pmatrix}}

\renewcommand{\epsilon}{\varepsilon}

\makeatletter
\newcommand*{\rom}[1]{\expandafter\@slowromancap\romannumeral #1@}
\makeatother

\begin{document}

\setcounter{tocdepth}{2}

\maketitle

\begin{abstract}
We use positive $S^1$-equivariant symplectic homology to define a sequence of symplectic capacities $c_k$ for star-shaped domains in $\R^{2n}$. These capacities are conjecturally equal to the Ekeland-Hofer capacities, but they satisfy axioms which allow them to be computed in many more examples. In particular, we give combinatorial formulas for the capacities $c_k$ of any ``convex toric domain'' or ``concave toric domain''. As an application, we determine optimal symplectic embeddings of a cube into any convex or concave toric domain. We also extend the capacities $c_k$ to functions of Liouville domains which are almost but not quite symplectic capacities.
\end{abstract}

\tableofcontents

\input{eh-intro.tex}

\input{eh-computations.tex}

\input{eh-capacities.tex}

\input{eh-equivariantsh.tex}

\input{eh-proofs1.tex}

\input{eh-transfer.tex}

\input{eh-proofs2.tex}

\bibliographystyle{alpha}
\bibliography{bibliography.bib}

\end{document}

%% file: eh-intro.tex
\section{Introduction}

\subsection{Symplectic capacities}
\label{sec:introsc}

Let $(X,\omega)$ and $(X',\omega')$ be symplectic manifolds of the same dimension, possibly noncompact or with boundary. A {\bf symplectic embedding\/} of $(X,\omega)$ into $(X',\omega')$ is a smooth embedding $\varphi:X\to X'$ such that $\varphi^{\star}\omega'=\omega$. A basic problem in symplectic geometry is to determine for which $(X,\omega)$ and $(X',\omega')$ a symplectic embedding exists. This is already a highly nontrivial question when the symplectic manifolds in question are domains in $\R^{2n}=\C^n$, with the restriction of the standard symplectic form.

Some basic examples of interest are as follows: If $a_1,\ldots,a_n>0$, define the {\bf ellipsoid\/}
\begin{equation}
\label{eqn:ellipsoiddef}
E(a_1,\ldots,a_n) = \left\{z\in\C^n\;\bigg|\;\sum_{i=1}^n\frac{\pi|z_i|^2}{a_i}\le 1\right\}
\end{equation}
and the {\bf polydisk\/}
\begin{equation}
\label{eqn:polydiskdef}
P(a_1,\ldots,a_n) = \left\{z\in\C^n\;\bigg|\;\pi|z_i|^2\le a_i, \; \; \forall i=1,\ldots,n\right\}.
\end{equation}
Also, define the {\bf ball\/} $B(a)=E(a,\ldots,a)$.
In the four-dimensional case ($n=2$), one can compute when an ellipsoid can be symplectically embedded into an ellipsoid or polydisk, although the answers are complicated; see e.g.\ \cite{ms,cgfs}. The question of when a four-dimensional polydisk can be symplectically embedded into a polydisk or an ellipsoid is only partially understood; for some of the latest results see \cite{beyond}. The analogous questions in higher dimensions are more complicated, and much less is understood; see e.g.\ \cite{guth,hind-kerman}.

In general, when studying symplectic embedding problems, one often obstructs the existence of symplectic embeddings using various kinds of symplectic capacities. Definitions of the latter term vary; in this paper we define a {\bf symplectic capacity\/} to be a function $c$ which assigns to each symplectic manifold $(X,\omega)$, possibly in some restricted class, a number $c(X,\omega)\in[0,\infty]$, satisfying the following axioms\footnote{One can also consider {\bf normalized\/} symplectic capacities, which satisfy the additional properties $c(B(1))=c(Z(1))=1$, where we define the {\bf cylinder\/} $Z(a)=P(a,\infty,\ldots,\infty)$. A strong version of a conjecture of Viterbo \cite{viterbo-conj} asserts that all normalized symplectic capacities agree on compact convex domains in $\R^{2n}$. For example, the Gromov width $c_{\op{Gr}}$ is normalized; the Ekeland-Hofer capacity $c_k^{\op{EH}}$ reviewed below is normalized when $k=1$, but not normalized when $n>1$ and $k>1$ since then $c_k^{\op{EH}}(B(1)) < c_k^{\op{EH}}(Z(1))$.}:
\begin{description}
\item{(Monotonicity)} If $(X,\omega)$ and $(X',\omega')$ have the same dimension, and if there exists a symplectic embedding $(X,\omega) \to (X',\omega')$, then $c(X,\omega) \le c(X',\omega')$.
\item{(Conformality)} If $r$ is a positive real number then $c(X,r\omega) = r c(X,\omega)$.
\end{description}

\noindent
For surveys about symplectic capacities, see e.g.\ \cite{chls,ostrover}.

One can easily define symplectic capacities in terms of symplectic embeddings to or from other symplectic manifolds. For example, the {\bf Gromov width\/} $c_{\op{Gr}}(X,\omega)$ is defined to be the supremum over $a$ such that the ball $B(a)$ can be symplectically embedded into $(X,\omega)$. This trivially satisfies the Monotonicity and Conformality axioms. A related example is the ``cube capacity'' studied in \S\ref{sec:cube} below. However, symplectic capacities defined like this are difficult to compute, since they are just encodings of nontrivial symplectic embedding questions.

Other symplectic capacities can be defined using Floer theory or related machinery, and these tend to be more computable. For example, the Ekeland-Hofer capacities \cite{EH2} are a nondecreasing sequence of capacities $c_k^{EH}$, indexed by a positive integer $k$, which are defined for compact star-shaped domains\footnote{In this paper, a ``domain'' in a Euclidean space is the closure of an open set.} in $\R^{2n}$. The Ekeland-Hofer capacities are defined using calculus of variations for the symplectic action functional on the loop space of $\R^{2n}$.  Computations of these capacities are known in a few examples. To state these, if $a_1,\ldots,a_n>0$, let $(M_k(a_1,\ldots,a_n))_{k=1,2,\ldots}$ denote the sequence of positive integer multiples of $a_1,\ldots,a_n$, arranged in nondecreasing order with repetitions. We then have:
\begin{itemize}
\item \cite[Prop.\ 4]{EH2} The Ekeland-Hofer capacities of an ellipsoid are given by
\begin{equation}
\label{eqn:EHellipsoid}
c_k^{\op{EH}}(E(a_1,\ldots,a_n)) = M_k(a_1,\ldots,a_n).
\end{equation}
\item \cite[Prop.\ 5]{EH2} The Ekeland-Hofer capacities of a polydisk are given by
\begin{equation}
\label{eqn:EHpolydisk}
c_k^{\op{EH}}(P(a_1,\ldots,a_n)) = k\cdot\min(a_1,\ldots,a_n).
\end{equation}
\item
Generalizing \eqref{eqn:EHpolydisk}, it is asserted in \cite[Eq.\ (3.8)]{chls} that if $X\subset\R^{2n}$ and $X'\subset\R^{2n'}$ are compact star-shaped domains, then for the (symplectic) Cartesian product $X\times X'\subset \R^{2(n+n')}$, we have
\begin{equation}
\label{eqn:EHproduct}
c_k^{\op{EH}}(X\times X') = \min_{i+j=k}\{c_i^{\op{EH}}(X)+c_j^{\op{EH}}(X')\},
\end{equation}
where $i$ and $j$ are nonnegative integers and we interpret $c_0^{\op{EH}}=0$.
\end{itemize}

More recently, embedded contact homology was used to define the {\em ECH capacities\/} of symplectic four-manifolds \cite{qech}. ECH capacities can be computed in many examples, such as four-dimensional ``concave toric domains'' \cite{concave} and ``convex toric domains'' \cite{concaveconvex,beyond}, defined in \S\ref{sec:examples} below. ECH capacities give sharp obstructions to symplectically embedding a four-dimensional ellipsoid into an ellipsoid \cite{mcduff-hc} or polydisk \cite{pnas}, or more generally a four-dimensional concave toric domain into a convex toric domain \cite{concaveconvex}. In some other situations, such as for some cases of symplectically embedding a four-dimensional polydisk into an ellipsoid, the ECH capacities do not give sharp obstructions, and the Ekeland-Hofer capacities are better;\footnote{The best currently known obstructions to symplectically embedding a four-dimensional polydisk into an ellipsoid are obtained using more refined information from embedded contact homology going beyond capacities \cite{beyond}.} see \cite[Rmk.\ 1.8]{qech}. The most significant weakness of ECH capacities is that they are only defined in four dimensions, and there is no known analogue of embedded contact homology in higher dimensions which might be used to define capacities.

In this paper we define a new sequence of symplectic capacities for domains in $\R^{2n}$ for any $n$. The idea is to imitate the definition of ECH capacities, but using positive $S^1$-equivariant symplectic homology in place of embedded contact homology. The resulting capacities conjecturally agree with the Ekeland-Hofer capacities, but they satisfy certain axioms which allow them to be computed in many more examples.

To state the axioms, define a {\bf nice star-shaped domain\/} in $\R^{2n}$ to be a compact $2n$-dimensional submanifold $X$ of $\R^{2n}=\C^n$ with smooth boundary $Y$, such that $Y$ is transverse to the radial vector field
\[
\rho = \frac{1}{2}\sum_{i=1}^n\left(x_i\frac{\partial}{\partial x_i} + y_i\frac{\partial}{\partial y_i}\right).
\]
In this case, the $1$-form
\begin{equation}
\label{eqn:lambda0}
\lambda_0= \frac{1}{2}\sum_{i=1}^n(x_i\,dy_i - y_i\,dx_i)
\end{equation}
on $\C^n$ restricts to a contact form on $Y$. If $\gamma$ is a Reeb orbit of ${\lambda_0}|_Y$, define its {\bf symplectic action\/} by
\[
\mc{A}(\gamma) = \int_{\gamma}\lambda_0 \in (0,\infty).
\]
If we further assume that ${\lambda_0}|_{Y}$ is nondegenerate, i.e.\ each Reeb orbit of ${\lambda_0}|_Y$ is nondegenerate, then each Reeb orbit $\gamma$ has a well-defined Conley-Zehnder\footnote{In the special case $n=1$, we have $Y\simeq S^1$, and we define $CZ(\gamma)$ to be twice the number of times that $\gamma$ covers $Y$.} index $\CZ(\gamma)\in\Z$. In this situation, if $k$ is a positive integer, define
\[
\mc{A}_k^- = \min\{\mc{A}(\gamma) | \CZ(\gamma) = 2k+n-1\}  \in (0,\infty)
\]
(we will see in a moment why this is finite),
and
\[
\mc{A}_k^+ = \sup\{\mc{A}(\gamma) | \CZ(\gamma) = 2k+n-1\} \in (0,\infty].
\]

In \S\ref{sec:defineck} we will define the new symplectic capacities $c_k$ for nice star-shaped domains in $\R^{2n}$ for each positive integer $k$.

\begin{theorem}
\label{thm:starshaped}
The capacities $c_k$ for nice star-shaped domains in $\R^{2n}$ satisfy the following axioms:
\begin{description}
\item{(Conformality)} If $X$ is a nice star-shaped domain in $\R^{2n}$ and $r$ is a positive real number, then $c(rX) = r^2 c(X)$.
\item{(Increasing)} $c_1(X) \le c_2(X) \le \cdots < \infty$.
\item{(Monotonicity)} If $X$ and $X'$ are nice star-shaped domains in $\R^{2n}$, and if there exists a symplectic embedding $X\to X'$, then $c_k(X)\le c_k(X')$ for all $k$.
\item{(Reeb Orbits)} If ${\lambda_0}|_{\partial X}$ is nondegenerate, then $c_k(X)=\mc{A}(\gamma)$ for some Reeb orbit $\gamma$ of $\lambda_0|_{\partial X}$ with $\CZ(\gamma) = 2k+n-1$. In particular,
\begin{equation}
\label{eqn:reeborbits}
\mc{A}_k^-(X) \le c_k(X) \le \mc{A}_k^+(X).
\end{equation}
\end{description}
\end{theorem}

\begin{remark}
The numbers $c_k$ are also discussed in \cite[\S3.2.1]{gg}, with applications to multiplicity results for simple Reeb orbits.
\end{remark}

\begin{remark}
\label{rem:extend}
We extend the capacities $c_k$ to functions of star-shaped domains which are not necessarily nice (such as polydisks) as follows: If $X$ is a star-shaped domain in $\R^{2n}$, then
\[
c_k(X) = \sup\{c_k(X')\}
\]
where the supremum is over nice star-shaped domains $X'$ in $\R^{2n}$ such that there exists a symplectic embedding $X'\to X$. It follows from Theorem~\ref{thm:starshaped} that this extended definition of $c_k$ continues to satisfy the first three axioms in Theorem~\ref{thm:starshaped}, and agrees with the previous definition when $X$ is a nice star-shaped domain.
\end{remark}

\subsection{Examples}
\label{sec:examples}

One can compute the capacities $c_k$ for many examples of star-shaped domains in $\R^{2n}$, using only the axioms in Theorem~\ref{thm:starshaped}.

To describe an important family of examples, let $\R^n_{\ge 0}$ denote the set of $x\in\R^n$ such that $x_i\ge 0$ for all $i=1,\ldots,n$. Define the moment map $\mu:\C^n\to\R^n_{\ge 0}$ by
\[
\mu(z_1,\ldots,z_n)=\pi(|z_1|^2,\ldots,|z_n|^2).
\]
If $\Omega$ is a domain in $\R^n_{\ge 0}$, define the {\bf toric domain\/}
\[
X_\Omega = \mu^{-1}(\Omega)\subset\C^n.
\]

We will study some special toric domains defined as follows. Given $\Omega\subset\R^n_{\ge 0}$, define
\[
\widehat{\Omega} = \left\{(x_1,\ldots,x_n)\in\R^n \;\big|\; (|x_1|,\ldots,|x_n|)\in\Omega\right\}.
\]

\begin{definition}
A {\bf convex toric domain\/} is a toric domain $X_\Omega$ such that $\widehat{\Omega}$ is a compact convex domain in $\R^n$.
\end{definition}

\begin{example}
If $n=2$, then $X_\Omega$ is a convex toric domain if and only if
\begin{equation}
\label{eqn:Omega2}
\Omega=\{(x_1,x_2)\mid 0\le x_1\le A,\; 0\le x_2\le g(x_1)\}
\end{equation}
where
\[
g:[0,A]\to \R_{\ge0}
\]
is a nonincreasing concave function. Some symplectic embedding problems involving these four-dimensional domains were studied in \cite{beyond}. A more general notion of ``convex toric domain'' in four dimensions, where $\Omega$ is convex but $\widehat{\Omega}$ is not required to be convex, is considered in \cite{concaveconvex}.
\end{example}

We now compute the capacities $c_k$ of a convex toric domain $X_\Omega$ in $\R^{2n}$. 
If $v\in\R^{n}_{\ge 0}$ is a vector with all components nonnegative, define\footnote{The reason for this notation is as follows. Let $\|\cdot\|_\Omega$ denote the norm on $\R^n$ whose unit ball is $\widehat{\Omega}$. 
Then in equation \eqref{eqn:dualnorm}, $\|\cdot\|_\Omega^*$ denotes the dual norm on $(\R^n)^*$, where the latter is identified with $\R^n$ using the Euclidean inner product.}
\begin{equation}
\label{eqn:dualnorm}
\|v\|_\Omega^*=\max\{\langle v,w\rangle \mid w\in\Omega\}
\end{equation}
where $\langle\cdot,\cdot\rangle$ denotes the Euclidean inner product. Let $\N$ denote the set of nonnegative integers.

\begin{theorem}
\label{thm:convex}
Suppose that $X_\Omega$ is a convex toric domain in $\R^{2n}$. Then
\begin{equation}
\label{eqn:convex}
c_k(X_\Omega) = \min\left\{\|v\|_\Omega^*\;\bigg|\; v=(v_1,\ldots,v_n)\in\N^n,\;\sum_{i=1}^nv_i=k\right\}.
\end{equation}
In fact, \eqref{eqn:convex} holds for any function $c_k$ defined on nice star-shaped domains in $\R^{2n}$ and satisfying the axioms in Theorem~\ref{thm:starshaped}, extended to general star-shaped domains as in Remark~\ref{rem:extend}.
\end{theorem}

\begin{example}
\label{ex:polydisk}
The polydisk $P(a_1,\ldots,a_n)$ is a convex toric domain $X_\Omega$, where $\Omega$ is the rectangle
\[
\Omega = \{x\in\R^n_{\ge 0} \mid x_i\le a_i,\;\forall i=1,\ldots,n\}.
\]
In this case
\[
\|v\|_\Omega^* = \sum_{i=1}^na_iv_i.
\]
It then follows from \eqref{eqn:convex} that
\[
c_k(P(a_1,\ldots,a_n)) = k\cdot\min\{a_1,\ldots,a_n\}.
\]
\end{example}

\begin{example}
\label{ex:ellipsoidconvex}
The ellipsoid $E(a_1,\ldots,a_n)$ is a convex toric domain $X_\Omega$, where $\Omega$ is the simplex
\[
\Omega=\left\{x\in\R^n_{\ge 0} \;\bigg|\; \sum_{i=1}^n\frac{x_i}{a_i} \le 1\right\}.
\]
In this case
\[
\|v\|_\Omega^* = \max_{i=1,\ldots,n}a_iv_i.
\]
Then \eqref{eqn:convex} gives
\[
c_k(E(a_1,\ldots,a_n)) = \min_{\sum_iv_i=k}\max_{i=1,\ldots,n}a_iv_i.
\]
It is a combinatorial exercise\footnote{To do the exercise, by a continuity argument we may assume that $a_i/a_j$ is irrational when $i\neq j$, so that the positive integer multiples of the numbers $a_i$ are distinct. If $v\in\N^n$ and $\sum_iv_i=k$, then the $k$ numbers $ma_i$ where $1\le i\le n$ and $1\le m\le v_i$ are distinct, which implies that the left hand side of \eqref{eqn:recoverellipsoid} is greater than or equal to the right hand side. To prove the reverse inequality, 
if $L=M_k(a_1,\ldots,a_n)$, then the numbers $v_i=\floor{L/a_i}$ satisfy $\sum_iv_i=k$ and $\max_{i=1,\ldots,n}a_iv_i=L$.} to check that
\begin{equation}
\label{eqn:recoverellipsoid}
\min_{\sum_iv_i=k}\max_{i=1,\ldots,n}a_iv_i = M_k(a_1,\ldots,a_n).
\end{equation}
We conclude that
\begin{equation}
\label{eqn:ckellipsoidintro}
c_k(E(a_1,\ldots,a_n) = M_k(a_1,\ldots,a_n).
\end{equation}
\end{example}

Comparing the above two examples with equations \eqref{eqn:EHellipsoid} and \eqref{eqn:EHpolydisk} suggests that our capacities $c_k$ may agree with the Ekeland-Hofer capacities $c_k^{\op{EH}}$:

\begin{conjecture}
\label{conj:eh}
Let $X$ be a compact star-shaped domain in $\R^{2n}$. Then
\[
c_k(X) = c_k^{\op{EH}}(X)
\]
for every positive integer $k$.
\end{conjecture}

\begin{remark}
More evidence for this conjecture: Theorem~\ref{thm:convex} implies that our capacities $c_k$ satisfy the Cartesian product property \eqref{eqn:EHproduct} in the special case when $X$ and $X'$ are convex toric domains. We do not know whether the capacities $c_k$ satisfy this property in general.
\end{remark}

We can also compute the capacities $c_k$ of another family of examples:

\begin{definition}
A {\bf concave toric domain\/} is a toric domain $X_\Omega$ where $\Omega$ is compact and $\R^n_{\ge 0}\setminus\Omega$ is convex.
\end{definition}

\begin{example}
If $n=2$, then $X_\Omega$ is a concave toric domain if and only if $\Omega$ is given by \eqref{eqn:Omega2} where $g:[0,A]\to\R_{\ge 0}$ is a convex function with $g(A)=0$. Some symplectic embedding problems involving these four-dimensional domains were studied in \cite{concave}.
\end{example}

\begin{remark} A domain in $X\subset \R^{2n}$ is both a convex toric domain and a concave toric domain if and only if $X$ is an ellipsoid \eqref{eqn:ellipsoiddef}.
\end{remark}

Suppose that $X_\Omega$ is a concave toric domain. Let $\Sigma$ denote the closure of the set $\partial\Omega\cap\R^n_{>0}$. Similarly to \eqref{eqn:dualnorm}, if $v\in\R^n_{\ge 0}$, define\footnote{Unlike \eqref{eqn:dualnorm}, the function $[\cdot]_\Omega$ is not a norm; instead it satisfies the reverse inequality $[v+v']_\Omega \ge [v]_\Omega + [v']_\Omega$.}
\begin{equation}
\label{eqn:antinorm}
[v]_\Omega = \min\left\{\langle v,w\rangle \;\big|\; w\in\Sigma\right\}.
\end{equation}

\begin{theorem}
\label{thm:concave}
If $X_\Omega$ is a concave toric domain in $\R^{2n}$, then
\begin{equation}
\label{eqn:concave}
c_k(X_\Omega) = \max\left\{[v]_\Omega \;\bigg|\; v\in\N^n_{>0},\;\sum_iv_i = k+n-1\right\}.
\end{equation}
In fact, \eqref{eqn:concave} holds for any function $c_k$ defined on nice star-shaped domains in $\R^{2n}$ and satisfying the axioms in Theorem~\ref{thm:starshaped}, extended to general star-shaped domains as in Remark~\ref{rem:extend}.
\end{theorem}

Note that in \eqref{eqn:concave}, all components of $v$ are required to be positive, while in \eqref{eqn:convex}, we only required that all components of $v$ be nonnegative.

\begin{example}
\label{ex:ellipsoidconcave}
Let us check that \eqref{eqn:concave} gives the correct answer when $X_\Omega$ is an ellipsoid $E(a_1,\ldots,a_n)$. 
Similarly to Example~\ref{ex:ellipsoidconvex}, we have
\[
[v]_\Omega = \min_{i=1,\ldots,n}a_iv_i.
\]
Thus, we need to check that
\begin{equation}
\label{eqn:ellipsoidconvex}
\max_{\sum_iv_i=k+n-1}\min_{i=1,\ldots,n}a_iv_i = M_k(a_1,\ldots,a_n)
\end{equation}
where, unlike Example~\ref{ex:ellipsoidconvex}, now all components of $v$ must be positive integers. This can be proved similarly to \eqref{eqn:recoverellipsoid}.
\end{example}

A quick application of Theorem~\ref{thm:concave}, pointed out by Schlenk \cite[Cor.\ 11.5]{schlenk}, is to compute the Gromov width of any concave toric domain\footnote{The four-dimensional case of this was shown using ECH capacities in \cite[Cor.\ 1.10]{concave}.}:

\begin{corollary}
If $X_\Omega$ is a concave toric domain in $\R^{2n}$, then
\[
c_{\op{Gr}}(X_\Omega) = \max\{a\mid B(a)\subset X_\Omega\}.
\]
\end{corollary}

\begin{proof}
Let $a_{\op{max}}$ denote the largest real number $a$ such that $B(a)\subset X_\Omega$.
By the definition of the Gromov width $c_{\op{Gr}}$, we have $c_{\op{Gr}}(X_\Omega)\ge a_{\op{max}}$. To prove the reverse inequality $c_{\op{Gr}}(X_\Omega)\le a_{\op{max}}$, suppose that there exists a symplectic embedding $B(a)\to X_\Omega$; we need to show that $a\le a_{\op{max}}$. By equation \eqref{eqn:ckellipsoidintro}, the monotonicity property of $c_1$, and Theorem~\ref{thm:concave}, we have
\[
\begin{split}
a &= c_1(B(a))\\
&\le c_1(X_\Omega)\\
 &= [(1,\ldots,1)]_\Omega\\
&= \min\left\{\sum_{i=1}^nw_i\;\bigg|\;w\in\Sigma\right\}\\
&= a_{\op{max}}.
\end{split}
\]
\end{proof}

\subsection{Application to cube capacities}
\label{sec:cube}

We now use the above results to solve some symplectic embedding problems where the domain is a cube.

Given $\delta>0$, define the {\bf cube\/}
\[
\Box_n(\delta) = P(\delta,\ldots,\delta) \subset \C^n.
\]
Equivalently,
\[
\Box_n(\delta) = \left\{z\in\C^n \;\bigg|\; \max_{i=1,\ldots,n}\left\{\pi|z_i|^2\right\}\le \delta\right\}.
\]

\begin{definition}
Given a $2n$-dimensional symplectic manifold $(X,\omega)$, define the {\bf cube capacity\/}
\[
c_\Box(X,\omega) = \sup\left\{\delta > 0 \; \mid \; \mbox{there exists a symplectic embedding $\Box_n(\delta) \longrightarrow (X,\omega)$}\right\}.
\]
\end{definition}

It is immediate from the definition that $c_\Box$ is a symplectic capacity.

\begin{theorem}
\label{thm:cube}
Let $X_\Omega\subset\C^n$ be a convex toric domain or a concave toric domain. Then
\[
c_\Box(X_\Omega) = \max\{\delta \mid (\delta,\ldots,\delta) \in \Omega\}.
\]
\end{theorem}

That is, $c_\Box(X_\Omega)$ is the largest $\delta$ such that $\Box_n(\delta)$ is a subset of $X_\Omega$; one cannot do better than this obvious symplectic embedding by inclusion.

Since the proof of Theorem~\ref{thm:cube} is short, we will give it now. We need to consider the non-disjoint union of $n$ symplectic cylinders,
\[
L_n(\delta) = \left\{z\in\C^n \;\bigg|\; \min_{i=1,\ldots,n}\left\{\pi|z_i|^2\right\}\le \delta\right\}.
\]

\begin{lemma}
\label{lem:ckz}
$c_k(L_n(\delta)) = \delta(k+n-1)$.
\end{lemma}

\begin{proof}
Observe that $L_n(\delta)=X_{\Omega_\delta}$ where
\[
\Omega_\delta=\left\{x\in\R^n_{\ge0} \;\bigg|\; \min_{i=1,\ldots,n}x_i\le \delta\right\}.
\]
As such, $\Omega_\delta$ is the union of a nested sequence of concave toric domains. By an exhaustion argument, the statement of Theorem~\ref{thm:concave} is valid for $X_{\Omega_\delta}$. Similarly to Example~\ref{ex:polydisk}, we have
\[
[v]_{\Omega_\delta} = \delta\sum_{i=1}^nv_i.
\]
The lemma then follows from equation \eqref{eqn:concave}.
\end{proof}

\begin{proposition}
\label{prop:cboxz}
$c_\Box(L_n(\delta))=\delta$.
\end{proposition}

\begin{proof}
We have $\Box_n(\delta)\subset L_n(\delta)$, so by the definition of $c_\Box$, it follows that $c_\Box(L_n(\delta))\ge \delta$.

To prove the reverse inequality $c_\Box(L_n(\delta))\le \delta$, suppose that there exists a symplectic embedding $\Box_n(\delta')\to L_n(\delta)$; we need to show that $\delta'\le\delta$.
By the Monotonicity property of the capacities $c_k$, we know that
\[
c_k(\Box_n(\delta')) \le c_k(L_n(\delta))
\]
for each positive integer $k$. By Example~\ref{ex:polydisk} and Lemma~\ref{lem:ckz}, this means that
\[
k\delta' \le \delta(k+n-1).
\]
Since this holds for arbitrarily large $k$, it follows that $\delta'\le\delta$ as desired.
\end{proof}

\begin{proof}[Proof of Theorem~\ref{thm:cube}.]
Let $\delta>0$ be the largest real number such that $(\delta,\ldots,\delta)\in\Omega$. It follows from the definitions of convex and concave toric domain that
\[
\Box_n(\delta) \subset X_\Omega \subset L_n(\delta).
\]
The first inclusion implies that $\delta\le c_\Box(X_\Omega)$ by the definition of $c_\Box$, while the second inclusion implies that $c_\Box(X_\Omega)\le\delta$ by Proposition~\ref{prop:cboxz}. Thus $c_\Box(X_\Omega) = \delta$.
\end{proof}

\begin{remark}
The proof of Theorem~\ref{thm:cube} shows more generally that any star-shaped domain $X\subset \C^n$ such that
\begin{equation}
\label{eqn:sandwich}
\Box_n(\delta) \subset X \subset L_n(\delta)
\end{equation}
satisfies $c_\Box(X)=\delta$.
\end{remark}

\begin{remark}
The proof of Theorem~\ref{thm:cube} also shows that if $X\subset\C^n$ is a star-shaped domain satisfying \eqref{eqn:sandwich}, then
\begin{equation}
\label{eqn:asymptotics}
\lim_{k\to\infty}\frac{c_k(X)}{k}
=
c_\Box(X).
\end{equation}
This is related to the following question of Cieliebak-Mohnke \cite{cm}.

Given a domain $X\subset\R^{2n}$, define the {\bf Lagrangian capacity\/} $c_L(X)$ to be the supremum over $A$ such that there exists an embedded Lagrangian torus $T\subset X$ such that the symplectic area of every map $(D^2,\partial D^2)\to (X,T)$ is an integer multiple of $A$. It is asked in \cite{cm} whether if $X\subset\R^{2n}$ is a convex domain then
\begin{equation}
\label{eqn:cm}
\lim_{k\to\infty}\frac{c_k^{EH}(X)}{k} = c_L(X).
\end{equation}
It is confirmed by \cite[Cor.\ 1.3]{cm} that \eqref{eqn:cm} holds when $X$ is a ball.

Observe that if $X$ is any domain in $\C^n$, then the Lagrangian capacity is related to the cube capacity by
\[
c_{\Box}(X)\le c_L(X),
\]
because if $\Box_n(\delta)$ symplectically embeds into $X$, then the restriction of this embedding maps the ``corner''
\[
\mu^{-1}(\delta,\ldots,\delta)\subset\Box_n(\delta)
\]
to a Lagrangian torus $T$ in $X$ such that the symplectic area of every disk with boundary on $T$ is an integer multiple of $\delta$. Thus the asymptotic result \eqref{eqn:asymptotics} implies that if $X\subset\C^n$ is a domain satisfying \eqref{eqn:sandwich}, then
\[
\lim_{k\to\infty}\frac{c_k(X)}{k} \le c_L(X).
\]
Assuming Conjecture~\ref{conj:eh}, this proves one inequality in \eqref{eqn:cm} for these examples.
\end{remark}

\subsection{Liouville domains}

Recall that a {\bf Liouville domain\/} is a pair $(X,\lambda)$ where $X$ is a compact manifold with boundary, $\lambda$ is a $1$-form on $X$ such that $d\lambda$ is symplectic, and $\lambda$ restricts to a contact form on $\partial X$, compatibly with the boundary orientation. For example, if $X$ is a nice star-shaped domain in $\R^{2n}$, then $(X,{\lambda_0}|_X)$ is a Liouville domain.

In \S\ref{sec:defineck} we extend the symplectic capacities $c_k$ for nice star-shaped domains to functions of Liouville domains. These are not quite capacities, because the Monotonicity property only holds under some restrictions:

\begin{definition}
\label{def:gle}
Let $(X,\lambda)$ and $(X',\lambda')$ be Liouville domains of the same dimension. A {\bf generalized Liouville embedding\/} $(X,\lambda)\to(X',\lambda')$ is a symplectic embedding $\varphi:(X,d\lambda)\to(X',d\lambda')$ such that
\[
\big[\left(\varphi^{\star}\lambda' - \lambda\right) \big|_{\partial X}\big] = 0 \in H^1(\partial X;\R).
\]
\end{definition}

Of course, if $H^1(\partial X;\R)=0$, for example if $X$ is a nice star-shaped domain in $\R^{2n}$, then every symplectic embedding is a generalized Liouville embedding.

\begin{theorem}
\label{thm:liouville}
The functions $c_k$ of Liouville domains satisfy the following axioms:
\begin{description}
\item{(Conformality)} If $(X,\lambda)$ is a Liouville domain and $r$ is a positive real number, then $c(X,r\lambda) = r c(X,\lambda)$.
\item{(Increasing)} $c_1(X,\lambda) \le c_2(X,\lambda) \le \cdots \le \infty$.
\item{(Restricted Monotonicity)} If there exists a generalized Liouville embedding $(X,\lambda)\to(X',\lambda')$,
then $c_k(X,\lambda) \le c_k(X',\lambda')$.
\item{(Contractible Reeb Orbits)} If $c_k(X,\lambda) < \infty$, then $c_k(X,\lambda) = \mc{A}(\gamma)$ for some Reeb orbit $\gamma$ of $\lambda|_{\partial X}$ which is contractible\footnote{Here $\mc{A}(\gamma)$ denotes the symplectic action of $\gamma$, which is defined by $\mc{A}(\gamma) = \int_\gamma\lambda$.} in $X$.
\end{description}
\end{theorem}

\begin{remark}
Monotonicity does not extend from generalized Liouville embeddings to arbitrary symplectic embeddings: in some cases there exists a symplectic embedding $(X,d\lambda) \to (X',d\lambda')$ even though $c_k(X,\lambda) > c_k(X',\lambda')$. For example, suppose that $T\subset X'$ is a Lagrangian torus. Let $\lambda_T$ denote the standard Liouville form on the cotangent bundle $T^*T$. By the Weinstein Lagrangian tubular neighborhood theorem, there is a symplectic embedding $(X,d\lambda)\to(X',d\lambda')$, where $X\subset T^*T$ is the unit disk bundle for some flat metric on $T$, and $\lambda = \lambda_T|_X$. Then $(X,\lambda)$ is a Liouville domain. But $\lambda|_{\partial X}$ has no Reeb orbits which are contractible in $X$, so by the Contractible Reeb Orbits axiom, $c_k(X,\lambda) = \infty$ for all $k$.

Note that the symplectic embedding $(X,d\lambda)\to (X',d\lambda')$ is a generalized Liouville embedding if and only if $T$ is an {\bf exact\/} Lagrangian torus in $(X',\lambda')$, that is $\lambda'|_T$ is exact. The Restricted Monotonicity axiom then tells us that {\em if $(X',\lambda')$ is a Liouville domain with $c_1(X',\lambda')<\infty$, then $(X',\lambda')$ does not contain any exact Lagrangian torus\/}.
\end{remark}

\begin{remark}
The functions $c_k$ are defined for disconnected Liouville domains. However, it follows from the definition in \S\ref{sec:defineck} that
\[
c_k\left(\coprod_{i=1}^m(X_i,\lambda_i)\right) = \max_{i=1,\ldots,m}c_k(X_i,\lambda_i).
\]
As a result, Restricted Monotonicity for embeddings of disconnected Liouville domains does not tell us anything more than it already does for their connected components.
\end{remark}

\begin{remark}
One can ask whether, by analogy with ECH capacities \cite[Prop.\ 1.5]{qech}, the existence of a generalized Liouville embedding $\coprod_{i=1}^m(X_i,\lambda_i)\to (X',\lambda')$ implies that
\begin{equation}
\label{eqn:disju}
\sum_{i=1}^mc_{k_i}(X_i,\lambda_i) \le c_{k_1+\cdots+k_m}(X',\lambda')
\end{equation}
for all positive integers $k_1,\cdots,k_m$. We have heuristic reasons to expect this when the $k_i$ are all multiples of $n-1$. However it is false more generally.

For example, in $2n$ dimensions, the Traynor trick \cite{traynor} can be used to symplectically embed the disjoint union of $n^2$ copies of the ball $B(1/2-\varepsilon)$ into the ball $B(1)$, for any $\varepsilon>0$. If \eqref{eqn:disju} is true with all $k_i=1$, then by \eqref{eqn:ckellipsoidintro} we obtain
\[
n^2(1/2-\varepsilon) \le n.
\]
But this is false when $n>2$ and $\varepsilon>0$ is small enough.
\end{remark}

\paragraph{Acknowledgments.} The first author thanks Mike Usher and Daniel Krashen for helpful discussions. The second author thanks Felix Schlenk for helpful discussions.

\paragraph{The rest of the paper.} In \S\ref{sec:computations} we prove Theorems~\ref{thm:convex} and \ref{thm:concave}, computing the capacities $c_k$ for convex and concave toric domains, using only the axioms in Theorem~\ref{thm:starshaped}. In \S\ref{sec:shproperties} we state the properties of positive $S^1$-equivariant symplectic homology and transfer morphisms that are needed to define the capacities $c_k$. In \S\ref{sec:defineck} we define the capacities $c_k$ and prove that they satisfy the axioms in Theorems~\ref{thm:starshaped} and \ref{thm:liouville}. In \S\ref{sec:equivariantsh} we review the definition of positive $S^1$-equivariant symplectic homology. In \S\ref{sec:proofs1}, we prove the properties of positive $S^1$-equivariant SH that are stated in \S\ref{sec:shproperties}. In \S\ref{section:transfer} we review the construction of transfer morphisms on positive $S^1$-equivariant SH. Finally, in \S\ref{sec:proofs2} we prove the properties of transfer morphisms that are stated in \S\ref{sec:shproperties}.

%% file: eh-computations.tex
\section{Computations of the capacities $c_k$}
\label{sec:computations}

We now prove Theorems~\ref{thm:convex} and \ref{thm:concave}, computing the capacities $c_k$ for convex and concave toric domains, using only the axioms in Theorem~\ref{thm:starshaped}.

\subsection{Computation for an ellipsoid}\label{sec:ellipsoid}

To prepare for the proofs of Theorems~\ref{thm:convex} and \ref{thm:concave}, we first compute the capacities $c_k$ for an ellipsoid (without using either of these theorems as in Example~\ref{ex:ellipsoidconvex} or \ref{ex:ellipsoidconcave}).

\begin{lemma}
\label{lem:ellipsoid}
The capacities $c_k$ of an ellipsoid are given by
\begin{equation}
\label{eqn:ellipsoid1}
c_k(E(a_1,\ldots,a_n)) = M_k(a_1,\ldots,a_n).
\end{equation}
\end{lemma}

\begin{proof}
By a continuity argument using the Monotonicity and Conformality axioms, cf.\ \cite[\S2.2]{concave}, to prove \eqref{eqn:ellipsoid1} we may assume that $a_i/a_j$ is irrational when $i\neq j$. In this case we can compute the capacities $c_k$ of the ellipsoid using the Reeb Orbits axiom in Theorem~\ref{thm:starshaped}. For this purpose, we now review how to compute the Reeb orbits on the boundary of the ellipsoid, their actions, and their Conley-Zehnder indices.

The Reeb vector field on the boundary of the ellipsoid $E(a_1,\ldots,a_n)$ is given by
\begin{equation}
\label{eqn:Reebellipsoid}
R = 2\pi\sum_{i=1}^n\frac{1}{a_i}\frac{\partial}{\partial\theta_i}
\end{equation}
where $\theta_i$ denotes the angular polar coordinate on the $i^{th}$ summand in $\C^n$. Since $a_i/a_j$ is irrational when $i\neq j$, it follows from \eqref{eqn:Reebellipsoid} that there are exactly $n$ simple Reeb orbits $\gamma_1,\ldots,\gamma_n$, where $\gamma_i$ denotes the circle where $z_j=0$ for $j\neq i$. We will also see below that $\lambda_0|_{\partial E(a_1,\ldots,a_n)}$ is nondegenerate.

It follows from \eqref{eqn:Reebellipsoid} that the actions of the simple Reeb orbits are given by
$
\mc{A}(\gamma_i) = a_i.
$
If $m$ is a positive integer, let $\gamma_i^m$ denote the $m^{th}$ iterate of $\gamma_i$; then this orbit has symplectic action
\begin{equation}
\label{eqn:agm}
\mc{A}(\gamma_i^m) = ma_i.
\end{equation}
Let $S$ denote the set of all such symplectic actions, i.e.\ the set of real numbers $ma_i$ where $m$ is a positive integer and $i\in\{1,\ldots,n\}.$ These are all distinct, by our assumption that $a_i/a_j$ is irrational when $i\neq j$.

We now compute the Conley-Zehnder indices of the Reeb orbits $\gamma_i^m$. Assume for the moment that $n>1$.
Recall that the Conley-Zehnder index of a contractible nondegenerate Reeb orbit $\gamma$ in a contact manifold $(Y,\lambda)$ with $c_1(\xi)|_{\pi_2(Y)}=0$ can be computed by the formula
\begin{equation}
\label{eqn:CZgeneral}
\CZ(\gamma) = \CZ_\tau(\gamma) + 2c_1(\gamma,\tau).
\end{equation}
Here $\tau$ is any (homotopy class of) symplectic trivialization of the restriction of the contact structure $\xi=\op{Ker}(\lambda)$ to $\gamma$; $\CZ_\tau(\gamma)$ denotes the Conley-Zehnder index of the path of symplectic matrices obtained by the linearized Reeb flow along $\gamma$ with respect to the trivialization $\tau$; and $c_1(\gamma,\tau)$ denotes the relative first Chern class with respect to $\tau$ of the pullback of $\xi$ to a disk $u$ bounded by $\gamma$, see \cite[\S3.2]{bn}.

In the present case where $Y=\partial E(a_1,\ldots,a_n)$, the contact structure $\xi$ on $\gamma_i^m$ is the sum of all of the $\C$ summands in $T\C^n=\C^n$ except for the $i^{th}$ summand. Let us use this identification to define the trivialization $\tau$. By \eqref{eqn:Reebellipsoid}, the linearized Reeb flow around $\gamma_i^m$ is the direct sum of rotation by angle $2\pi m a_i/a_j$ in the $j^{th}$ summand for each $j\neq i$. It follows that
\[
\CZ_\tau(\gamma_i^m) = \sum_{j\neq i}\left(2\floor{\frac{ma_i}{a_j}} + 1\right).
\]
On the other hand,
\begin{equation}
\label{eqn:c1ellipsoid}
c_1(\gamma_i^m,\tau) = m,
\end{equation}
essentially because the Hopf fibration over $S^2$ has Euler number $1$. Putting this all together, we obtain
\[
\begin{split}
\CZ(\gamma_i^m) & = 2m + \sum_{j\neq i}\left(2\floor{\frac{ma_i}{a_j}} + 1\right)\\
&= n-1 + 2\sum_{j=1}^n\floor{\frac{ma_i}{a_j}}.
\end{split}
\]
Thus
\begin{equation}
\label{eqn:czgm}
\CZ(\gamma_i^m)
= n-1 + 2\left|\left\{L\in S \mid L\le ma_i\right\}\right|.
\end{equation}
It follows from \eqref{eqn:agm} and \eqref{eqn:czgm} that
\begin{equation}
\label{eqn:mist}
\CZ(\gamma_i^m) = n-1+2k
\Longleftrightarrow 
\mc{A}(\gamma_i^m) = M_k(a_1,\ldots,a_n).
\end{equation}
Note that this also holds when $n=1$, by our convention in \S\ref{sec:introsc}.

In conclusion, it follows from \eqref{eqn:mist} that for each positive integer $k$ we have
\[
\mc{A}_k^-(E(a_1,\ldots,a_n)) = M_k(a_1,\ldots,a_n) =  \mc{A}_k^+(E(a_1,\ldots,a_n)).
\]
The lemma now follows from the Reeb Orbits axiom \eqref{eqn:reeborbits}.
\end{proof}

\begin{remark}
A useful equivalent version of \eqref{eqn:ellipsoid1} is
\begin{equation}
\label{eqn:ellipsoid2}
c_k(E(a_1,\ldots,a_n)) = \min\left\{ L \;\bigg| \; \sum_{i=1}^n\floor{\frac{L}{a_i}}\ge k\right\}.
\end{equation}
\end{remark}

\begin{remark}
\label{rem:exhaustion}
Lemma~\ref{lem:ellipsoid}, in the form \eqref{eqn:ellipsoid2}, extends to the case where some of the numbers $a_i$ are infinite, by an exhaustion argument.
\end{remark}

\subsection{Computation for convex toric domains}

We now prove Theorem~\ref{thm:convex}. We first prove that the left hand side of \eqref{eqn:convex} is less than or equal to the right hand side:

\begin{lemma}
\label{lem:convexeasy}
If $X_\Omega$ is a convex toric domain in $\R^{2n}$ then
\[
c_k(X_\Omega) \le \min\left\{\|v\|_\Omega^*\;\bigg|\; v=(v_1,\ldots,v_n)\in\N^n,\;\sum_{i=1}^nv_i=k\right\}.
\]
\end{lemma}

\begin{proof}
Let $v=(v_1,\ldots,v_n)\in\N^n$ with
\begin{equation}
\label{eqn:sumvi}
\sum_{i=1}^nv_i=k;
\end{equation}
we need to show that $c_k(X_\Omega) \le \|v\|_\Omega^*$. Write $L=\|v\|_\Omega^*$. By the definition \eqref{eqn:dualnorm} of $\|\cdot\|_\Omega^*$, we have $L=\langle v,w\rangle$, where $w\in\Omega$ such that $\langle v,w\rangle$ is maximal. Define
\begin{equation}
\label{eqn:Omega'}
\Omega'= \left\{x\in\R^n_{\ge0} \;\big|\; \langle v,x\rangle \le L \right\}.
\end{equation}
Then by maximality of $\langle v,w\rangle$ we have $\Omega\subset\Omega'$. By the Monotonicity axiom for the capacity $c_k$, it follows that
\[
c_k(X_\Omega) \le c_k(X_{\Omega'}).
\]
Thus it suffices to show that $c_k(X_{\Omega'})\le L$.

To do so, suppose first that $v_i>0$ for all $i=1,\ldots,n$. Then $X_{\Omega'}$ is an ellipsoid,
\[
X_{\Omega'} = E\left(\frac{L}{v_1},\ldots,\frac{L}{v_n}\right).
\]
By equation \eqref{eqn:ellipsoid2}, we have
\begin{equation}
\label{eqn:ckxo'}
c_k(X_{\Omega'}) = \min\left\{L'\;\bigg|\; \sum_{i=1}^n\floor{\frac{L'}{L/v_i}} \ge k\right\}.
\end{equation}
Since the $v_i$ are integers, by equation \eqref{eqn:sumvi} we have
\[
L'=L \;\Longrightarrow\; \sum_{i=1}^n\floor{\frac{L'}{L/v_i}} = k.
\]
It follows from this and \eqref{eqn:ckxo'} that $c_k(X_{\Omega'})\le L$ as desired (in fact this is an equality).

The above calculation extends to the case where some of the components $v_i$ are zero by Remark~\ref{rem:exhaustion}.
\end{proof}

We now use a different argument to prove the reverse inequality which completes the proof of Theorem~\ref{thm:convex}:

\begin{lemma}
\label{lem:convexhard}
If $X_\Omega$ is a convex toric domain in $\R^{2n}$ then
\begin{equation}
\label{eqn:convexhard}
c_k(X_\Omega) \ge \min\left\{\|v\|_\Omega^*\;\bigg|\; v=(v_1,\ldots,v_n)\in\N^n,\;\sum_{i=1}^nv_i=k\right\}.
\end{equation}
\end{lemma}

\begin{proof}
If $n=1$, then the result follows immediately from Lemma~\ref{lem:ellipsoid}; thus we may assume that $n>1$.

To start, we perturb $\Omega$ to have some additional properties that will be useful.
It follows from the Conformality and Monotonicity axioms that the left hand side of \eqref{eqn:convexhard} is continuous with respect to the Hausdorff metric on compact sets $\Omega$, as in \cite[Lem.\ 2.3]{concave}. The right hand side is also continuous with respect to the Hausdorff metric as in \cite[Lem.\ 2.4]{concave}. As a result, we may assume the following, where $\Sigma$ denotes the closure of the set $\partial\Omega\cap\R^n_{>0}$:
\begin{description}
\item{(i)}
$\Sigma$ is a smooth hypersurface in $\R^n$.
\item{(ii)}
The Gauss map $G:\Sigma\to S^{n-1}$ is a smooth embedding, and $\partial X_\Omega$ is a smooth hypersurface in $\R^{2n}$. In particular, $X_\Omega$ is a nice star-shaped domain.
\item{(iii)}
If $w\in \Sigma$ and if $w_i=0$ for some $i$, then the $i^{th}$ component of $G(w)$ is positive and small with respect to $k$.
\end{description}
We now prove \eqref{eqn:convexhard} in four steps.

{\em Step 1.\/} 
We first compute the Reeb vector field on $\partial X_\Omega=\mu^{-1}(\Sigma)$.

Let $w\in\Sigma$ and let $z\in\mu^{-1}(w)$. Also, write $G(w) = (\nu_1,\ldots,\nu_n)$. Observe that
\[
\sum_i\nu_iw_i=\|G(w)\|_\Omega^*.
\]

We now define local coordinates on a neighborhood of $z$ in $\C^n$ as follows. For $i=1,\ldots,n$, let $\C_i$ denote the $i^{th}$ summand in $\C^n$. If $z_i=0$, then we use the standard coordinates $x_i$ and $y_i$ on $\C_i$. If $z_i\neq 0$, then on $\C_i$ we use local coordinates $\mu_i$ and $\theta_i$, where $\mu_i=\pi(x_i^2+y_i^2)$, and $\theta_i$ is the angular polar coordinate.

In these coordinates, the standard Liouville form \eqref{eqn:lambda0} is given by
\[
\lambda_0 = \frac{1}{2}\sum_{w_i=0}\left(x_i\,dy_i - y_i\,dx_i\right) + \frac{1}{2\pi}\sum_{w_i\neq 0}\mu_i\,d\theta_i.
\]
Also, the tangent space to $\partial X_\Omega$ at $z$ is described by
\[
T_z\partial X_\Omega = \bigoplus_{w_i=0}\C_i \oplus\left\{\sum_{w_i\neq 0}\left(a_i\frac{\partial}{\partial\mu_i} + b_i\frac{\partial}{\partial\theta_i}\right) \;\bigg|\; \sum_{w_i\neq 0}\nu_ia_i=0\right\}.
\]
It follows from the above three equations that the Reeb vector field at $z$ is given by
\begin{equation}
\label{eqn:rvf}
R = \frac{2\pi}{\|G(w)\|_\Omega^*}\sum_{w_i\neq 0}\nu_i\frac{\partial}{\partial \theta_i}.
\end{equation}
For future reference, we also note that the contact structure $\xi$ at $z$ is given by
\begin{equation}
\label{eqn:xiz}
\xi_z = \bigoplus_{w_i=0}\C_i \oplus\left\{\sum_{w_i\neq 0}\left(a_i\frac{\partial}{\partial\mu_i} + b_i\frac{\partial}{\partial\theta_i}\right) \;\bigg|\; \sum_{w_i\neq 0}\nu_ia_i=0,\;\sum_{w_i\neq 0}w_ib_i=0\right\}.
\end{equation}

{\em Step 2.\/} We now compute the Reeb orbits and their basic properties.

It is convenient here to define a (discontinuous) modification $\widetilde{G}:\Sigma\to \R^n$ of the Gauss map $G$ by setting a component of the output to zero whenever the corresponding component of the input is zero. That is, for $i=1,\ldots,n$ we define
\begin{equation}
\label{eqn:gtilde}
\widetilde{G}(w)_i = \left\{\begin{array}{cl}  G(w)_i, & w_i\neq 0,\\
0, & w_i=0.
\end{array}\right.
\end{equation}

Continuing the discussion from Step 1, observe from \eqref{eqn:rvf} that the Reeb vector field $R$ is tangent to $\mu^{-1}(w)$. Let $Z(w)$ denote the number of components of $w$ that are equal to zero; then $\mu^{-1}(w)$ is a torus of dimension $n-Z(w)$. It follows from \eqref{eqn:rvf} that if $\widetilde{G}(w)$ is a scalar multiple of an integer vector, then $\mu^{-1}(w)$ is foliated by an $(n-Z(w)-1)$-dimensional Morse-Bott family of Reeb orbits; otherwise $\mu^{-1}(w)$ contains no Reeb orbits.

Let $V$ denote the set of nonnegative integer vectors $v$ such that $v$ is a scalar multiple of an element $\tilde{v}$ of the image of the modified Gauss map $\widetilde{G}$. Given $v\in V$, let $d(v)$ denote the greatest common divisor of the components of $v$. Let $\mc{P}(v)$ denote the set of $d(v)$-fold covers of simple Reeb orbits in the torus $\mu^{-1}\left(\widetilde{G}^{-1}\left(\tilde{v}\right)\right)$. Then it follows from the above discussion that the set of Reeb orbits on $\partial X_\Omega$ equals $\sqcup_{v\in V}\mc{P}(v)$. Moreover, condition (iii) above implies that $v\in V$ whenever $\sum_iv_i\le k$.

Equation \eqref{eqn:rvf} implies that each Reeb orbit $\gamma\in\mc{P}(v)$ has symplectic action
\[
\mc{A}(\gamma) = \|v\|_{\Omega}^*.
\]

Also, we can define a trivialization $\tau$ of $\xi|_\gamma$ from \eqref{eqn:xiz}, identifying $\xi_z$ for each $z\in\gamma$ with 
a codimension two subspace of $\R^{2n}$ with coordinates $x_i,y_i$ for each $i$ with $w_i=0$, and coordinates $a_i,b_i$ for each $i$ with $w_i\neq 0$.  Then, similarly to \eqref{eqn:c1ellipsoid}, we have
\begin{equation}
\label{eqn:c1convex}
c_1(\gamma,\tau) = \sum_{i=1}^nv_i.
\end{equation}

{\em Step 3.\/} We now approximate the convex toric domain $X_\Omega$ by a nice star-shaped domain $X'$ such that $\lambda_0|_{\partial X'}$ is nondegenerate.

Given $v\in V$ with $d(v)=1$, one can perturb $\partial X_\Omega$ in a neighborhood of the $n-Z(v)$ dimensional torus swept out by the Reeb orbits in $\mc{P}(v)$, using a Morse function $f$ on the $n-Z(v)-1$ dimensional torus $\mc{P}(v)$, to resolve the Morse-Bott family $\mc{P}(v)$ into a finite set of nondegenerate Reeb orbits corresponding to the critical points of $f$ (possibly together with some additional Reeb orbits of much larger symplectic action). Owing to the strict convexity of $\Sigma$, each such nondegenerate Reeb orbit $\gamma$ will have Conley-Zehnder index with respect to the above trivialization $\tau$ in the range
\begin{equation}
\label{eqn:CZrange}
Z(v) \le \CZ_\tau(\gamma) \le n-1.
\end{equation}
It then follows from \eqref{eqn:c1convex} that
\begin{equation}
\label{eqn:CZperturb}
Z(v) + 2\sum_{i=1}^nv_i \le \CZ(\gamma) \le  n-1+ 2\sum_{i=1}^nv_i.
\end{equation}
In particular,
\begin{equation}
\label{eqn:inpa1}
\CZ(\gamma) = 2k + n - 1
\Longrightarrow
k \le \sum_{i=1}^nv_i \le k + \frac{n-1-Z(v)}{2}.
\end{equation}
Moreover, even if we drop the assumption that $d(v)=1$, then after perturbing the orbits in $\mc{P}(v/d(v))$ as above, the family $\mc{P}(v)$ will still be replaced by nondegenerate orbits each satisfying \eqref{eqn:CZperturb} (possibly together with additional Reeb orbits of much larger symplectic action), as long as $d(v)$ is not too large with respect to the perturbation.

Now choose $\epsilon>0$ small and choose
\[
R>\max\left\{\|v\|_\Omega^*\;\bigg|\; v\in\N^n,\; \sum_iv_i \le k+\frac{n-1}{2}\right\}.
\]
We can then perturb $X_\Omega$ to a nice star-shaped domain $X'$ with $\lambda_0|_{\partial X'}$ nondegenerate such that for each $v\in V$ with $\|v\|_\Omega^*<R$, the Morse-Bott family $\mc{P}(v)$ is perturbed as above; each nondegenerate orbit $\gamma$ arising from each such $\mc{P}(v)$ has symplectic action satisfying
\begin{equation}
\label{eqn:inpa2}
\mc{A}(\gamma) \ge \|v\|_\Omega^*-\epsilon;
\end{equation}
and there are no other Reeb orbits of symplectic action less than $R$.

{\em Step 4.\/} We now put together the above inequalities to complete the proof.

It follows from \eqref{eqn:inpa1} and \eqref{eqn:inpa2} that
\[
\mc{A}_k^-(X') \ge \min\left\{\|v\|_\Omega^*-\epsilon 
\;\bigg|\; v\in\N^n,\;k \le \sum_{i=1}^nv_i \le k + \frac{n-1-Z(v)}{2} \right\}.
\]
Thus by the Reeb Orbits axiom \eqref{eqn:reeborbits}, we have
\[
c_k(X') \ge \min\left\{\|v\|_\Omega^*-\epsilon 
\;\bigg|\; v\in\N^n,\;k \le \sum_{i=1}^nv_i \le k + \frac{n-1-Z(v)}{2} \right\}.
\]
Taking $\epsilon\to0$ and a sequence of perturbations $X'$ converging in $C^0$ to $X_\Omega$, and using Conformality and Monotonicity as at the beginning of the proof of this lemma, we obtain
\[
c_k(X_\Omega) \ge \min\left\{\|v\|_\Omega^* 
\;\bigg|\; v\in\N^n,\;k \le \sum_{i=1}^nv_i \le k + \frac{n-1-Z(v)}{2} \right\}.
\]

In fact, in the above minimum, we can restrict attention to $v$ with $\sum_iv_i=k$, because if $\sum_iv_i>k$, then we can decrease some components of $v$ to obtain a new vector $v'\in\N^n$ with $\sum_iv_i'=k$, and by equation \eqref{eqn:dualnorm} we will have $\|v'\|_\Omega^*\le \|v\|_\Omega^*$. This completes the proof of \eqref{eqn:convexhard}.
\end{proof}

\subsection{Computation for concave toric domains}

We now prove Theorem~\ref{thm:concave}. The proof is very similar to the above proof of Theorem~\ref{thm:convex}, but with the direction of some inequalities switched, and other slight changes.

\begin{lemma}
If $X_\Omega$ is a concave toric domain in $\R^{2n}$, then
\[
c_k(X_\Omega) \ge \max\left\{[v]_\Omega \;\bigg|\; v\in\N^n_{>0},\;\sum_iv_i = k+n-1\right\}.
\]
\end{lemma}

\begin{proof}
Let $v\in\N^n_{>0}$ with
\begin{equation}
\label{eqn:kn1}
\sum_iv_i=k+n-1;
\end{equation}
we need to show that $c_k(X_\Omega) \ge [v]_\Omega$.

Write $L=[v]_\Omega$. By the definition \eqref{eqn:antinorm}, we have
$L=\langle v,w\rangle$, where $w\in\Sigma$ is such that $\langle v,w\rangle$ is  minimal. If we define $\Omega'$ as in \eqref{eqn:Omega'}, then by minimality of $\langle v,w\rangle$ we have $\Omega'\subset\Omega$. By monotonicity of the capacity $c_k$, we then have
\[
c_k(X_{\Omega'})\le c_k(X_{\Omega}).
\]
So it suffices to show that $c_k(X_{\Omega'})\ge L$. We again have equation \eqref{eqn:ckxo'}, namely
\[
c_k(X_{\Omega'}) = \min\left\{L'\;\bigg|\; \sum_{i=1}^n\floor{\frac{L'}{L/v_i}} \ge k\right\}.
\]
Since the $v_i$ are integers, by equation \eqref{eqn:kn1} we have
\[
L'<L \Longrightarrow \sum_{i=1}^n\floor{\frac{L'}{L/v_i}} \le k-1.
\]
It follows that $c_k(X_{\Omega'})\ge L$ as desired (in fact this is an equality).
\end{proof}

\begin{lemma}
If $X_\Omega$ is a concave toric domain in $\R^{2n}$, then
\begin{equation}
\label{eqn:concavehard}
c_k(X_\Omega) \le \max\left\{[v]_\Omega \;\bigg|\; v\in\N^n_{>0},\;\sum_iv_i = k+n-1\right\}.
\end{equation}
\end{lemma}

\begin{proof}
As in the proof of Lemma~\ref{lem:convexhard}, we may assume that $n>1$ and that:
\begin{description}
\item{(i)}
$\Sigma$ is a smooth hypersurface in $\R^n$.
\item{(ii)}
The Gauss map $G:\Sigma\to S^{n-1}$ is a smooth embedding, and $\partial X_\Omega$ is a smooth hypersurface in $\R^{2n}$, so that $X_\Omega$ is a nice star-shaped domain.
\item{(iii)}
If $w\in \Sigma$ and $w_i=0$ for some $i$, then $G(w)$ is close (with respect to $k$) to the set of $\nu\in S^{n-1}$ such that $\nu_j=0$ whenever $w_j\neq 0$.
\end{description}
Similarly to \eqref{eqn:rvf}, the Reeb vector field again preserves each torus $\mu^{-1}(w)$, on which now
\begin{equation}
\label{eqn:nrvf}
R = \frac{2\pi}{[G(w)]_\Omega}\sum_{w_i\neq 0}\nu_i\frac{\partial}{\partial \theta_i}
\end{equation}
where $G(w)=(\nu_1,\ldots,\nu_n)$. Continuing to define the modified Gauss map $\widetilde{G}$ by equation \eqref{eqn:gtilde}, and defining $V$ and $\mc{P}(v)$ as before, it follows that the set of Reeb orbits on $\partial X_\Omega$ is again given by $\sqcup_{v\in V}\mc{P}(v)$. Condition (iii) above implies that $v\in V$ whenever $\sum_iv_i$ is not too large, and equation \eqref{eqn:nrvf} implies that each Reeb orbit in $\mc{P}(v)$ has action $[v]_\Omega$.

As in Step 3 of the proof of Lemma~\ref{lem:convexhard}, we can perturb the concave toric domain $X_\Omega$ to a nice star-shaped domain $X'$ such that the contact form $\lambda_0|_{X'}$ is nondegenerate; up to large symplectic action, the Reeb orbits come from the tori $\mc{P}(v)$ where $\sum_iv_i$ is not too large; and a Reeb orbit $\gamma$ coming from $\mc{P}(v)$ has action
\begin{equation}
\label{eqn:step3a}
\mc{A}(\gamma) \le [v]_\Omega + \epsilon
\end{equation}
where $\epsilon>0$ can be chosen arbitrarily small.

If any component of $v$ is zero, then the Conley-Zehnder index of $\gamma$ will be very large, by condition (iii) above. Otherwise,
to compute the Conley-Zehnder index of $\gamma$, we use a homotopy class $\tau$ of trivialization of $\xi_\gamma$ defined as in the proof of Lemma~\ref{lem:convexhard}. Equation \eqref{eqn:c1convex} still holds, while the inequalities \eqref{eqn:CZrange} are replaced by
\[
0 \le -\CZ_\tau(\gamma) \le n-1.
\]
(Here the sign of $\CZ_\tau(\gamma)$ is switched because $\Sigma$ is concave instead of convex.)
Thus we obtain
\[
1-n+2\sum_{i=1}^nv_i  \le \CZ(\gamma) \le 2\sum_{i=1}^nv_i.
\]
In particular, we obtain
\begin{equation}
\label{eqn:step3b}
\CZ(\gamma) = 2k+n-1 \Longrightarrow k+\frac{n-1}{2} \le \sum_{i=1}^nv_i \le k + n - 1.
\end{equation}

It follows from \eqref{eqn:step3a} and \eqref{eqn:step3b} that
\[
\mc{A}_k^+(X') \le \max\left\{[v]_\Omega + \epsilon \;\bigg|\; v\in\N^n_{>0},\; k+\frac{n-1}{2} \le \sum_{i=1}^nv_i \le k + n - 1\right\}.
\]
As in Step 4 of the proof of Lemma~\ref{lem:convexhard}, we deduce that
\[
c_k(X_\Omega) \le \max\left\{[v]_\Omega \;\bigg|\; v\in\N^n_{>0}, \; k+\frac{n-1}{2} \le \sum_{i=1}^nv_i \le k + n - 1\right\}
\]
In the above maximum, we can restrict attention to $v$ with $\sum_iv_i=k+n-1$, since increasing some components of $v$ will not decrease $[v]_\Omega$. This completes the proof of \eqref{eqn:concavehard}.
\end{proof}

%% file: eh-capacities.tex
\section{Input from positive $S^1$-equivariant symplectic homology}
\label{sec:shproperties}

We now state the properties of positive $S^1$-equivariant symplectic homology, and transfer morphisms defined on it, that are needed to define the capacities $c_k$ and establish their basic properties. These properties are stated in Propositions~\ref{prop:ch} and \ref{prop:transfer} below, which are proved in \S\ref{sec:proofs1} and \S\ref{sec:proofs2} respectively.

We say that a Liouville domain $(X,\lambda)$ is {\bf nondegenerate\/} if the contact form $\lambda|_{\partial X}$ is nondegenerate. In this case we can define the {\bf positive $S^1$-equivariant symplectic homology} $SH^{S^1,+}(X,\lambda)$, see \S\ref{sec:SHS1+}. This is a $\Q$-module\footnote{It is also possible to define positive $S^1$-equivariant symplectic homology with integer coefficients. However the torsion in the latter is not relevant to the construction of the capacities $c_k$, and it will simplify our discussion to discard it.}.
To simplify notation, we often denote $SH^{S^1,+}(X,\lambda)$ by $CH(X,\lambda)$ below\footnote{The reason for this notation is that positive $S^1$-equivariant symplectic homology can be regarded as a substitute for linearized contact homology which can be defined without transversality difficulties \cite[\S3.2]{bo}.}.

\begin{proposition}
\label{prop:ch}
The positive $S^1$-equivariant symplectic homology $CH(X,\lambda)$ has the following properties:
\begin{description}
\item{(Free homotopy classes)}
$CH(X,\lambda)$ has a direct sum decomposition
\[
CH(X,\lambda) = \bigoplus_\Gamma CH(X,\lambda,\Gamma)
\]
where $\Gamma$ ranges over free homotopy classes of loops in $X$. We let $CH(X,\lambda,0)$ denote the summand corresponding to contractible loops in $X$.
\item{(Action filtration)}
For each $L\in\R$, there is a $\Q$-module $CH^L(X,\lambda,\Gamma)$ which is an invariant of $(X,\lambda,\Gamma)$. If $L_1<L_2$, then there is a well-defined map
\begin{equation}
\label{eqn:il1l2}
\imath_{L_2,L_1}: CH^{L_1}(X,\lambda,\Gamma) \longrightarrow CH^{L_2}(X,\lambda,\Gamma).
\end{equation}
These maps form a directed system, and we have the direct limit
\[
\lim_{L\to\infty} CH^L(X,\lambda,\Gamma) = CH(X,\lambda,\Gamma).
\]
We denote the resulting map $CH^L(X,\lambda,\Gamma) \to CH(X,\lambda,\Gamma)$ by $\imath_L$. We write $CH^L(X,\lambda) = \bigoplus_\Gamma CH^L(X,\lambda,\Gamma)$.
\item{($U$ map)} There is a distinguished map
\[
U:CH(X,\lambda,\Gamma)\longrightarrow CH(X,\lambda,\Gamma),
\]
which respects the action filtration in the following sense: For each $L\in\R$ there is a map
\[
U_L:CH^L(X,\lambda,\Gamma) \longrightarrow CH^L(X,\lambda,\Gamma).
\]
If $L_1<L_2$ then $U_{L_2}\circ \imath_{L_2,L_1} = \imath_{L_2,L_1}\circ U_{L_1}$. The map $U$ is the direct limit of the maps $U_L$, i.e.
\begin{equation}
\label{eqn:iLU}
\imath_L\circ U_L = U\circ \imath_L.
\end{equation}
\item{(Reeb Orbits)}
If $L_1<L_2$, and if there does not exist a Reeb orbit $\gamma$ of $\lambda|_{\partial X}$ in the free homotopy class $\Gamma$ with action $\mc{A}(\gamma)\in (L_1,L_2]$, then the map \eqref{eqn:il1l2} is an isomorphism.
\item{($\delta$ map)} There is a distinguished map
\[
\delta: CH(X,\lambda,\Gamma) \longrightarrow H_*(X,\partial X;\Q) \tensor H_*(BS^1;\Q)
\]
which vanishes whenever $\Gamma\neq 0$.
\item{(Scaling)} If $r$ is a positive real number then there are canonical isomorphisms
\[
\begin{split}
CH(X,\lambda,\Gamma) &\stackrel{\simeq}{\longrightarrow} CH(X,r\lambda,\Gamma),\\
CH^L(X,\lambda,\Gamma) &\stackrel{\simeq}{\longrightarrow} CH^{rL}(X,r\lambda,\Gamma)
\end{split}
\]
which commute with all of the above maps.
\item{(Star-Shaped Domains)} If $X$ is a nice star-shaped domain in $\R^{2n}$ and $\lambda_0$ is the restriction of the standard Liouville form \eqref{eqn:lambda0}, then:
\begin{description}
\item{(i)}
$CH(X,\lambda_0)$ and $CH^L(X,\lambda_0)$ have canonical $\Z$ gradings. With respect to this grading, we have
\begin{equation}
\label{eqn:sscomputation}
CH_*(X,\lambda_0) \simeq \left\{\begin{array}{cl} \Q, & \mbox{if $*\in n+1+2\N$},\\
0, & \mbox{otherwise}.
\end{array}\right.
\end{equation}
\item{(ii)}
The map $\delta$ sends a generator of $CH_{n-1+2k}(X,\lambda_0)$ to a generator of $H_{2n}(X,\partial X;\Q)$ tensor a generator of $H_{2k-2}(BS^1;\Q)$.
\item{(iii)}
The $U$ map has degree $-2$ and is an isomorphism
\[
CH_*(X,\lambda_0) \stackrel{\simeq}{\longrightarrow} CH_{*-2}(X,\lambda_0),
\]
except when $*=n+1$. 
\item{(iv)}
If $\lambda_0|_{\partial X}$ is nondegenerate and has no Reeb orbit $\gamma$ with $\mc{A}(\gamma)\in(L_1,L_2]$ and $\op{CZ}(\gamma)=n-1+2k$, then the map 
\[
\imath_{L_2,L_1}: CH_{n-1+2k}^{L_1}(X,\lambda_0) \to CH_{n-1+2k}^{L_2}(X,\lambda_0)
\]
is surjective.
\end{description}
\end{description}
\end{proposition}

\begin{remark}
One can presumably refine the ``Reeb Orbits'' property to show that in the nondegenerate case, $CH^L(X,\lambda,\Gamma)$ is the homology of a chain complex (with noncanonical differential) which is generated by the good Reeb orbits $\gamma$ of $\lambda|_{\partial X}$ in the free homotopy class $\Gamma$ with symplectic action $\mc{A}(\gamma)\le L$. (A Reeb orbit $\gamma$ is called {\bf bad\/} if it is an even multiple cover of a Reeb orbit $\gamma'$ such that the Conley-Zehnder indices of $\gamma$ and $\gamma'$ have opposite parity; otherwise it is called {\bf good}.) Moreover, if $L_1<L_2$, then one can take the differential for $L_1$ to be the restriction of the differential for $L_2$, and the map $\imath_{L_2,L_1}$ is induced by the inclusion of chain complexes. This is shown in \cite[Prop.\ 3.3]{gg} using a different definition of equivariant symplectic homology.
\end{remark}

Now suppose that $(X',\lambda')$ is another nondegenerate Liouville domain and $\varphi:(X,\lambda)\to(X',\lambda')$ is a generalized Liouville embedding (see Definition~\ref{def:gle}) with $\varphi(X)\subset\op{int}(X')$. One can then define a {\bf transfer morphism}
\[
\Phi: CH(X',\lambda') \longrightarrow CH(X,\lambda),
\]
see \S\ref{section:transfer}.

\begin{proposition}
\label{prop:transfer}
The transfer morphism $\Phi$ has the following properties:
\begin{description}
\item{(Action)} $\Phi$ respects the action filtration in the following sense: For each $L\in\R$ there are distinguished maps
\[
\Phi^L:CH^L(X',\lambda') \longrightarrow CH^L(X,\lambda)
\]
such that if $L_1<L_2$ then
\begin{equation}
\label{eqn:filtran}
\Phi^{L_2}\circ \imath_{L_2,L_1} = \imath_{L_2,L_1}\circ \Phi^{L_1},
\end{equation}
and $\Phi$ is the direct limit of the maps $\Phi^L$, i.e.
\begin{equation}
\label{eqn:iLP}
\imath_L\circ \Phi^L = \Phi\circ \imath_L.
\end{equation}
\item{(Commutativity with $U$)}
For each $L\in\R$, the diagram
\begin{equation}
\label{eqn:commU}
\begin{CD}
CH^L(X',\lambda') @>{\Phi^L}>> CH^L(X,\lambda)\\
@VV{U^L}V @VV{U^L}V\\
CH^L(X',\lambda') @>{\Phi^L}>> CH^L(X,\lambda)
\end{CD}
\end{equation}
commutes.
\item{(Commutativity with $\delta$)} The diagram
\begin{equation}
\label{eqn:commdelta}
\begin{CD}
CH(X',\lambda') @>{\Phi}>> CH(X,\lambda)\\
@VV{\delta}V @VV{\delta}V \\
H_*(X',\partial X';\Q)\tensor H_*(BS^1;\Q) @>{\rho\tensor 1}>> H_*(X,\partial X;\Q)\tensor H_*(BS^1;\Q)
\end{CD}
\end{equation}
commutes. Here $\rho:H_*(X',\partial X';\Q) \to H_*(X,\partial X;\Q)$ denotes the composition
\[
H_*(X',\partial X';\Q) \longrightarrow H_*(X',X'\setminus \varphi(\op{int}(X));\Q) \stackrel{\simeq}{\longrightarrow} H_*(\varphi(X),\varphi(\partial X);\Q) = H_*(X,\partial X;\Q)
\]
where the first map is the map on relative homology induced by the triple $(X',X'\setminus \varphi(\op{int}(X)), \partial X')$, and the second map is excision. 
\end{description}
\end{proposition}

\section{Definition of the capacities $c_k$}
\label{sec:defineck}

\subsection{Nondegenerate Liouville domains}

We first define the capacities $c_k$ for nondegenerate Liouville domains, imitating the definition of ECH capacities in \cite[Def.\ 4.3]{qech}.

\begin{definition}
\label{def:cknondeg}
Let $(X,\lambda)$ be a nondegenerate Liouville domain and let $k$ be a positive integer. Define
\[
c_k(X,\lambda)\in(0,\infty]
\]
to be the infimum over $L$ such that there exists $\alpha\in CH^L(X,\lambda)$ satisfying
\begin{equation}
\label{eqn:alpha}
\delta U^{k-1}\imath_L\alpha = [X]\tensor [\op{pt}] \in H_*(X,\partial X)\tensor H_*(BS^1).
\end{equation}
\end{definition}

We now show that the function $c_k$ satisfies most of the axioms of Theorem~\ref{thm:liouville} (restricted to nondegenerate Liouville domains):

\begin{lemma}
\label{lem:cknondeg}
\begin{description}
\item{(a)}
If $(X,\lambda)$ is a nondegenerate Liouville domain and $r>0$, then
\[
c_k(X,r\lambda)=rc_k(X,\lambda).
\]
\item{(b)} If $(X,\lambda)$ is a nondegenerate Liouville domain and $k>1$ then
\[
c_{k-1}(X,\lambda)\le c_k(X,\lambda).
\]
\item{(c)} If $(X,\lambda)$ and $(X',\lambda')$ are nondegenerate Liouville domains, and if there exists a generalized Liouville embedding $\varphi: (X,\lambda)\to (X',\lambda')$ with $\varphi(X)\subset\op{int}(X')$, then
\[
c_k(X,\lambda) \le c_k(X',\lambda').
\]
\item{(d)} If $(X,\lambda)$ is a nondegenerate Liouville domain, and if $c_k(X,\lambda)<\infty$, then $c_k(X,\lambda) = \mc{A}(\gamma)$ for some Reeb orbit $\gamma$ of $\lambda|_{\partial X}$ which is contractible in $X$.
\end{description}
\end{lemma}

\begin{proof}
(a) This follows from the Scaling axiom in Proposition~\ref{prop:ch}.

(b) Suppose that $\alpha\in CH^L(X,\lambda)$ satisfies \eqref{eqn:alpha}. We need to show that there exists $\alpha'\in CH^L(X,\lambda)$ such that
\[
\delta U^{k-2}\imath_L\alpha' = [X]\tensor [\op{pt}].
\]
By equation \eqref{eqn:iLU}, we can take $\alpha'=U_L\alpha$.

(c) Suppose that $\alpha'\in CH^L(X',\lambda')$ satisfies
\begin{equation}
\label{eqn:alpha'}
\delta U^{k-1}\imath_L\alpha' = [X']\tensor [\op{pt}].
\end{equation}
We need to show that there exists $\alpha\in CH^L(X,\lambda)$ satisfying
\[
\delta U^{k-1}\imath_L\alpha = [X]\tensor [\op{pt}].
\]
We claim that we can take $\alpha=\Phi^L\alpha'$ where $\Phi^L$ is the filtered transfer map from Proposition~\ref{prop:transfer}(a). To see this, we observe that
\[
\begin{split}
\delta U^{k-1} \imath_L \Phi^L \alpha' &= \delta\imath_L U_L^{k-1}\Phi^L\alpha'\\
&= \delta\imath_L\Phi^LU_L^{k-1}\alpha'\\
&= \delta\Phi\imath_LU_L^{k-1}\alpha'\\
&= (\rho\tensor 1)\delta U^{k-1}\imath_L\alpha'\\
&= (\rho\tensor 1)([X']\tensor[\op{pt}])\\
&= [X]\tensor[\op{pt}].
\end{split}
\]
Here the first equality holds by \eqref{eqn:iLU}, the second equality follows from \eqref{eqn:commU}, the third equality holds by \eqref{eqn:iLP}, the fourth equality uses \eqref{eqn:commdelta} and \eqref{eqn:iLU} again, and the fifth equality follows from the hypothesis \eqref{eqn:alpha'}.

(d) Suppose that $c_k(X,\lambda) = L < \infty$. Suppose to get a contradiction that there is no Reeb orbit of action $L$ which is contractible in $X$. Since $\lambda|_{\partial X}$ is nondegenerate, there are only finitely many Reeb orbits of action less than $2L$. It follows that we can find $\epsilon>0$ such that there is no Reeb orbit which is contractible in $X$ and has action in the interval $[L-\epsilon,L+\epsilon]$, and such that there exists $\alpha_+\in CH^{L+\epsilon}(X,\lambda)$ with
\[
\delta U^{k-1}\imath_{L+\epsilon}\alpha_+ = [X]\tensor [\op{pt}].
\]
By the last part of the ``$\delta$ map'' property in Proposition~\ref{prop:ch}, we can assume that $\alpha_+\in CH^{L+\epsilon}(X,\lambda,0)$. By the ``Reeb Orbits'' property, there exists $\alpha_-\in CH^{L-\epsilon}(X,\lambda,0)$ with $\imath_{L+\epsilon,L-\epsilon}\alpha_-=\alpha_+$. It follows that
\[
\delta U^{k-1}\imath_{L-\epsilon}\alpha_- = [X]\tensor [\op{pt}].
\]
This implies that $c_k(X,\lambda)\le L-\epsilon$, which is the desired contradiction.
\end{proof}

\subsection{Arbitrary Liouville domains}
\label{sec:ald}

We now extend the definition of $c_k$ to an arbitrary Liouville domain $(X,\lambda)$. To do so, we use the following procedure to perturb a possibly degenerate Liouville domain to a nondegenerate one.

First recall that there is a distinguished Liouville vector field $V$ on $X$ characterized by $\imath_Vd\lambda=\lambda$. Write $Y=\partial X$. The flow of $V$ then defines a smooth embedding
\begin{equation}
\label{eqn:embedding}
(-\infty,0]\times Y \longrightarrow X,
\end{equation}
sending $\{0\}\times Y$ to $Y$ in the obvious way,
such that if $\rho$ denotes the $(-\infty,0]$ coordinate, then $\partial_\rho$  is mapped to the vector field $V$. This embedding pulls back the Liouville form $\lambda$ on $X$ to the $1$-form $e^\rho(\lambda|_Y)$ on $(-\infty,0]\times Y$.
The {\bf completion\/} of $(X,\lambda)$ is the pair $(\widehat{X},\widehat{\lambda})$ defined as follows. First,
\[
\widehat{X} = X\cup_Y([0,\infty)\times Y),
\] 
glued using the identification \eqref{eqn:embedding}. Observe that $\widehat{X}$ has a subset which is identified with $\R\times Y$, and we denote the $\R$ coordinate on this subset by $\rho$. The $1$-form $\lambda$ on $X$ then extends to a unique $1$-form $\widehat{\lambda}$ on $\widehat{X}$ which agrees with $e^\rho(\lambda|_Y)$ on $\R\times Y$.

Now if $f:Y\to\R$ is any smooth function, define a new Liouville domain $(X_f,\lambda_f)$, where
\[
X_f = \widehat{X} \setminus \{(\rho,y)\in\R\times Y \mid \rho>f(y)\},
\]
and $\lambda_f$ is the restriction of $\widehat{\lambda}$ to $X_f$. For example, if $f\equiv 0$, then $(X_f,\lambda_f)=(X,\lambda)$. In general, there is a canonical identification
\[
\begin{split}
Y &\longrightarrow \partial X_f,\\
y &\longmapsto (f(y),y)\in\R\times Y.
\end{split}
\]
Under this identification,
\[
\lambda_f|_{\partial X_f} = e^f\lambda|_{Y}.
\]

We now consider $c_k$ of nondegenerate perturbations of a possibly degenerate Liouville domain.

\begin{lemma} (cf. \cite[Lem.\ 3.5]{qech})
\label{lem:makessense}
\begin{description}
\item{(a)}
If $(X,\lambda)$ is any Liouville domain, then
\begin{equation}
\label{eqn:ckarb}
\sup_{f_- < 0}c_k(X_{f_-},\lambda_{f_-}) = \inf_{f_+ > 0} c_k(X_{f_+},\lambda_{f_+}).
\end{equation}
Here the supremum and infimum are taken over functions $f_-:Y\to(-\infty,0)$ and $f_+:Y\to(0,\infty)$ respectively such that the contact form $e^{f_\pm}(\lambda|_Y)$ is nondegenerate.
\item{(b)}
If $(X,\lambda)$ is nondegenerate, then the supremum and infimum in \eqref{eqn:ckarb} agree with $c_k(X,\lambda)$.
\end{description}
\end{lemma}

As a result of Lemma~\ref{lem:makessense}, it makes sense to extend Definition~\ref{def:cknondeg} as follows:

\begin{definition}
\label{def:ckarb}
If $(X,\lambda)$ is any Liouville domain, define $c_k(X,\lambda)$ to be the supremum and infimum in \eqref{eqn:ckarb}.
\end{definition}

The proof of Lemma~\ref{lem:makessense} will use the following simple fact:

\begin{lemma}
\label{lem:conftrick}
If $(X_{f_-},\lambda_{f_-})$ is nondegenerate, and if $f_+=f_-+\epsilon$ for some $\epsilon\in\R$, then $(X_{f_+},\lambda_{f_+})$ is also nondegenerate and
\[
c_k(X_{f_+},\lambda_{f_+}) = e^\epsilon c_k(X_{f_-},\lambda_{f_-}).
\]
\end{lemma}

\begin{proof}
We have that $(X_{f_+},\lambda_{f_+})$ is nondegenerate because scaling the contact form on the boundary by a constant (in this case $e^\epsilon$) scales the Reeb vector field and preserves nondegeneracy.

The time $\epsilon$ flow of the Liouville vector field $V$ on $\widehat{X}$ restricts to a diffeomorphism $X_{f_-}\to X_{f_+}$ which pulls back $\lambda_{f_+}$ to $e^\epsilon \lambda_{f_-}$.  It follows that
\[
c_k(X_{f_+},\lambda_{f_+}) = c_k(X_{f_-},e^\epsilon\lambda_{f_-}) = e^\epsilon c_k(X_{f_-},\lambda_{f_-}),
\]
where the second equality holds by the conformality in Lemma~\ref{lem:cknondeg}(a).
\end{proof}

\begin{proof}[Proof of Lemma~\ref{lem:makessense}.]
(a) If $f_-,f_+:Y\to\R$ satisfy $f_- < f_+$, then inclusion defines a Liouville embedding $\varphi:(X_{f_-},\lambda_{f_-})\to (X_{f_+},\lambda_{f_+})$ with $\varphi(X_{f_-})\subset \op{int}(X_{f_+})$. It then follows from the monotonicity in Lemma~\ref{lem:cknondeg}(c) that
\[
c_k(X_{f_-},\lambda_{f_-}) \le c_k(X_{f_+},\lambda_{f_+}).
\]
This shows that the left hand side of \eqref{eqn:ckarb} is less than or equal to the right hand side.

To prove the reverse inequality, for any $\epsilon>0$ we can find a function $f_+:Y\to(0,\epsilon)$ such that the contact form $e^{f_+}(\lambda|_Y)$ is nondegenerate. Now define $f_-:Y\to(-\epsilon,0)$ by $f_-=f_+-\epsilon$. By Lemma~\ref{lem:conftrick} we have
\[
c_k(X_{f_+},\lambda_{f_+}) = e^\epsilon c_k(X_{f_-},\lambda_{f_-}).
\]
It follows that
\[
\inf_{f_+ > 0} c_k(X_{f_+},\lambda_{f_+}) \le e^\epsilon \sup_{f_- < 0}c_k(X_{f_-},\lambda_{f_-}).
\]
Since $\epsilon>0$ was arbitrary, we conclude that the right hand side of \eqref{eqn:ckarb} is less than or equal to the left hand side.

(b) In this case, for any $\epsilon>0$ we can take $f_\pm=\pm\epsilon$ in \eqref{eqn:ckarb}, so using Lemma~\ref{lem:conftrick} we have
\[
\begin{split}
\sup_{f_- < 0}c_k(X_{f_-},\lambda_{f_-}) &\ge c_k(X_{-\epsilon},\lambda_{-\epsilon}) = e^{-\epsilon} c_k(X,\lambda),\\
\inf_{f_+ > 0} c_k(X_{f_+},\lambda_{f_+}) &\le c_k(X_\epsilon,\lambda_\epsilon) = e^{\epsilon} c_k(X,\lambda).
\end{split}
\]
Taking $\epsilon\to 0$, we obtain
\[
\sup_{f_- < 0}c_k(X_{f_-},\lambda_{f_-}) \ge  c_k(X,\lambda) \ge
\inf_{f_+ > 0} c_k(X_{f_+},\lambda_{f_+}).
\]
The result now follows from the first half of part (a).
\end{proof}

\begin{proposition}
\label{prop:liouville}
The function $c_k$ of Liouville domains satisfies the Conformality, Increasing, Restricted Monotonicity, and Contractible Reeb Orbit axioms in Theorem~\ref{thm:liouville}.
\end{proposition}

\begin{proof}
The Conformality and Increasing axioms follow immediately from the corresponding properties in Lemma~\ref{lem:cknondeg}(a),(b).

To prove the Restricted Monotonicity property, suppose that there exists a generalized Liouville embedding $\varphi:(X,\lambda) \to (X',\lambda')$. Let $f_-:\partial X\to (-\infty,0)$ and $f_+:\partial X'\to (0,\infty)$ be smooth functions such that $(X_{f_-},\lambda_{f_-})$ and $(X'_{f_+},\lambda'_{f_+})$ are nondegenerate. Then we can restrict $\varphi$ to $X_{f_-}$, and compose with the inclusion $X'\to X'_{f_+}$, to obtain a generalized Liouville embedding $\widetilde{\varphi}: (X_{f_-},\lambda_{f_-})\to (X'_{f_+},\lambda'_{f_+})$ with $\widetilde{\varphi}(X_{f_-})\subset \op{int}(X'_{f_+})$. By the monotonicity in Lemma~\ref{lem:cknondeg}(c), we have
\[
c_k(X_{f_-},\lambda_{f_-}) \le c_k(X'_{f_+},\lambda'_{f_+}).
\]
It follows that
\[
\sup_{f_-<0}c_k(X_{f_-},\lambda_{f_-}) \le \inf_{f_+>0}c_k(X'_{f_+},\lambda'_{f_+}),
\]
which means that $c_k(X,\lambda) \le c_k(X',\lambda')$.

The Contractible Reeb Orbit axiom follows from the corresponding property in Lemma~\ref{lem:cknondeg}(d) and a compactness argument.
\end{proof}

\subsection{Nice star-shaped domains}

We now study $c_k$ of nice star-shaped domains and complete the proofs of Theorems~\ref{thm:starshaped} and \ref{thm:liouville}.

\begin{proof}[Proof of Theorems~\ref{thm:starshaped} and \ref{thm:liouville}.]
By Proposition~\ref{prop:liouville}, it is enough to show the functions $c_k$, restricted to nice star-shaped domains, satisfy the axioms in Theorem~\ref{thm:starshaped}. The Conformality and Increasing axioms follow immediately from the corresponding properties in Proposition~\ref{prop:liouville}. The Monotonicity axiom in Theorem~\ref{thm:starshaped} follows from the Restricted Monotonicity axiom in Proposition~\ref{prop:liouville}, because if $X$ and $X'$ are nice star-shaped domains in $\R^{2n}$, then any symplectic embedding $X\to X'$ is automatically a generalized Liouville embedding since $H_1(\partial X)=0$. Finally, the Reeb Orbits axiom follows from Lemma~\ref{lem:nssd}(b) below.
\end{proof}

\begin{lemma}
\label{lem:nssd}
Let $X$ be a nice star-shaped domain in $\R^{2n}$. Suppose that $\lambda_0|_{\partial X}$ is nondegenerate. Then:
\begin{description}
\item{(a)}
$c_k(X,\lambda_0)$ is the infimum over $L$ such that the degree $n-1+2k$ summand in $CH(X,\lambda_0)$ is in the image of the map $\imath_L:CH^L(X,\lambda_0)\to CH(X,\lambda_0)$.
\item{(b)}
$c_k(X,\lambda_0)=\mc{A}(\gamma)$ for some Reeb orbit $\gamma$ of $\lambda_0|_{\partial X}$ with $\CZ(\gamma)=n-1+2k$.
\end{description}
\end{lemma}

\begin{proof}
(a) This follows immediately from the definition of $c_k$ and the Star-Shaped Domains property in Proposition~\ref{prop:ch}.

(b) This follows from (a), similarly to the proof of Lemma~\ref{lem:cknondeg}(d).
\end{proof}

\begin{remark}
If one is only interested in nice star-shaped domains, then one can take the characterization of $c_k$ in Lemma~\ref{lem:nssd}(a) as the definition of $c_k$.
\end{remark}

%% file: eh-equivariantsh.tex
\section{Definition of positive $S^1$-equivariant SH}
\label{sec:equivariantsh}

Our remaining goal is to prove Propositions~\ref{prop:ch} and \ref{prop:transfer}. We now review what we need to know about positive $S^1$-equivariant symplectic homology for this purpose.

(Positive) symplectic homology was developed by Viterbo \cite{V}, using works of Cieliebak, Floer, and Hofer \cite{FH,CFH}. The $S^1$-equivariant version of (positive) symplectic homology was originally defined by Viterbo \cite{V}, and an alternate definition using family Floer homology was given by Bourgeois-Oancea \cite[\S2.2]{bo}, following a suggestion of Seidel \cite{Seidel}. We will use the family Floer homology definition here, because it is more amenable to computations. We follow the treatment in \cite{gutt}, with some minor tweaks which do not affect the results.

We will only consider (positive, $S^1$-equivariant) symplectic homology for Liouville domains, even though it can be defined for more general compact symplectic manifolds with contact-type boundary. We restrict to Liouville domains in order to be able to define transfer morphisms.

\subsection{Symplectic homology}

Let $(X,\lambda)$ be a Liouville domain with boundary $Y$. Let $R_\lambda$ denote the Reeb vector field associated to $\lambda$ on $Y$. Below, let $\Spec(Y,\lambda)$ denote the set of periods of Reeb orbits, and let $\epsilon=\frac{1}{2}\min\Spec(Y,\lambda)$.

Recall from \S\ref{sec:ald} that the completion $(\widehat{X},\widehat{\lambda})$ of $(X,\lambda)$ is defined by
\[
\widehat{X}:=X\cup\bigl([0,\infty)\times Y\bigr) \quad\textrm{and}\quad\widehat{\lambda}:=
\begin{cases}
	\lambda &\textrm{on } X,\\
	e^\rho \lambda|_Y &\textrm{on }[0,\infty)\times Y
\end{cases}
\]
where $\rho$ denotes the $[0,\infty)$ coordinate. Write $\widehat{\omega}=d\widehat{\lambda}$. Consider a $1$-periodic Hamiltonian on $\widehat{X}$, i.e.\ a smooth function
\[
H:S^1\times\widehat{X}\longrightarrow\R
\]
where $S^1=\R/\Z$. Such a function $H$ determines a vector field $X_H^\theta$ on $\widehat{X}$ for each $\theta\in S^1$, defined by $\widehat{\omega}(X_H^\theta,\cdot) = dH(\theta,\cdot)$. Let $\Per(H)$ denote the set of $1$-periodic orbits of $X_H$, i.e.\ smooth maps $\gamma:S^1\to\widehat{X}$ satisfying the equation $\gamma'(\theta) = X_H^\theta\big(\gamma(\theta)\big)$.

\begin{definition}
\label{def:Hstd}
An {\bf admissible Hamiltonian\/} is a smooth function $H : S^{1}\times\widehat{X} \rightarrow \R$ satisfying the following conditions:
\begin{description}
\item{(1)}
The restriction of $H$ to $S^1\times X$ is negative, autonomous (i.e.\ $S^1$-independent), and $C^2$-small (so that there are no non-constant 1-periodic orbits). Furthermore,
\begin{equation}
\label{eqn:HSpec}
H > -\epsilon
\end{equation}
on $S^1\times X$.
\item{(2)}
There exists $\rho_{0} \geq 0$ such that on $S^1 \times [\rho_0,\infty) \times Y$ we have
\begin{equation}
\label{eqn:limitingslope}
H(\theta,\rho,y) = \beta e^{\rho} + \beta'
\end{equation}
with $0<\beta\notin\Spec(Y,\lambda)$  and $\beta'\in\R$. The constant $\beta$ is called the {\bf limiting slope\/} of $H$.
\item{(3)}
There exists a small, strictly convex, increasing function $h:[1,e^{\rho_0}]\to\R$ such that on $S^1\times[0,\rho_0]\times Y$, the function $H$ is $C^{2}$-close to the function sending $(\theta,\rho,x)\mapsto h(e^\rho)$. The precise sense of ``small'' and ``close'' that we need here is explained in Remarks~\ref{rem:orbitsofHstand} and \ref{rem:hsmall}.
\item{(4)}
The Hamiltonian $H$ is nondegenerate, i.e.\ all $1$-periodic orbits of $X_{H}$ are nondegenerate.
\end{description}
We denote the set of admissible Hamiltonians by $\Hstd$.
\end{definition}

\begin{remark}
\label{rem:orbitsofHstand}
Condition (1) implies that the only $1$-periodic orbits of $X_H$ in $X$ are constants; they correspond to critical points of $H$.

The significance of condition (2) is as follows. On $S^{1}\times [0,\infty)\times Y$, for a Hamiltonian of the form $H_1(\theta,\rho,y)=h_1(e^{\rho})$, we have
\[
X_{H_1}^{\theta}(\rho,y) = -h'_1(e^{\rho})R_{\lambda}(y).
\]
Hence for such a Hamitonian $H_1$ with $h_1$ increasing, a $1$-periodic orbit of $X_{H_1}$ maps to a level $\{\rho\}\times Y$, and the image of its projection to $Y$ is the image of a (not necessarily simple) periodic Reeb orbit of period $h'_1(e^{\rho})$. In particular, condition (2) implies that there is no $1$-periodic orbit of $X_H$ in $[\rho_0,\infty) \times Y$.

Condition (3) ensures that for any non-constant $1$-periodic orbit $\gamma_H$ of $X_H$, there exists a (not necessarily simple) periodic Reeb orbit $\gamma$ of period $T<\beta$ such that the image of $\gamma_H$ is close to the image of $\gamma$ in $\{ \rho \}\times  Y$ where  $T=h'(e^{\rho})$.
\end{remark}

\begin{definition}
\label{def:admJ}
An $S^1$-family of almost complex structures $J:S^1\to\End(T\widehat{X})$ is {\bf admissible\/} if it satisfies the following conditions:
\begin{itemize}
\item
$J^\theta$ is $\widehat{\omega}$-compatible for each $\theta\in S^1$.
\item
There exists $\rho_1\ge 0$ such that on $[\rho_1,\infty)\times Y$, the almost complex structure $J^\theta$ does not depend on $\theta$, is invariant under translation of $\rho$, sends $\xi$ to itself compatibly with $d\lambda$, and satisfies
\begin{equation}
\label{eqn:JReeb}
J^\theta (\partial_{\rho})=R_{\lambda}.
\end{equation}
\end{itemize}
We denote the set of all admissible $J$ by $\mathcal{J}$.
\end{definition}

Given $J\in\mathcal{J}$, and $\gamma_-,\gamma_+\in\Per(H)$, let $\widehat{\M}(\gamma_-,\gamma_+;J)$ denote the set of maps
\[
u: \R\times S^1\longrightarrow \widehat{X}
\]
satisfying Floer's equation
\begin{equation}
\label{eqn:Floer}
\frac{\partial u}{\partial s}(s,\theta) + J^{\theta}\bigl(u(s,\theta)\bigr)\biggl(\frac{\partial u}{\partial \theta}(s,\theta) - X_H^{\theta}\bigl(u(s,\theta)\bigr)\biggr)=0
\end{equation}
as well as the asymptotic conditions
\[
\lim_{s\to\pm\infty}u(s,\cdot) = \gamma_\pm.
\]
If $J$ is generic and $u\in\M(\gamma_-,\gamma_+;J)$, then $\widehat{\M}(\gamma_-,\gamma_+;J)$ is a manifold near $u$ whose dimension is the Fredholm index of $u$ defined by
\[
\op{ind}(u) = \op{CZ}_\tau(\gamma_+) - \op{CZ}_\tau(\gamma_-).
\]
Here $\op{CZ}_\tau$ denotes the Conley-Zehnder index computed using trivializations $\tau$ of $\gamma_\pm^{\star}T\widehat{X}$ that extend to a trivialization of $u^{\star}T\widehat{X}$. Note that $\R$ acts on $\widehat{\M}(\gamma_-,\gamma_+;J)$ by translation of the domain; we denote the quotient by $\M(\gamma_-,\gamma_+;J)$.

\begin{definition}	
Let $H\in\Hstd$, and let $J\in\mathcal{J}$ be generic.
Define the Floer chain complex $(CF(H,J),\partial)$ as follows. The chain module $CF(H,J)$ is the free $\Q$-module\footnote{It is also possible to use $\Z$ coefficients here, but we will use $\Q$ coefficients in order to later establish the Reeb Orbits property in Proposition~\ref{prop:ch}, which leads to the Reeb Orbits property of the capacities $c_k$. In special cases when the Conley-Zehnder index of a $1$-periodic orbit is unambiguously defined, for example when all $1$-periodic orbits are contractible and $c_1(TX)|_{\pi_2(X)}=0$, the chain complex is graded by minus the Conley-Zehnder index.} generated by the set of $1$-periodic orbits $\Per(H)$. If $\gamma_-,\gamma_+\in\Per(H)$, then the coefficient of $\gamma_+$ in $\partial\gamma_-$ is obtained by counting Fredholm index $1$ points in $\M(\gamma_-,\gamma_+;J)$ with signs determined by a system of coherent orientations as in \cite{FH2}. (The chain complexes for different choices of coherent orientations are canonically isomorphic.)
\end{definition}
	
Let $HF(H,J)$ denote the homology of the chain complex $(CF(H,J),\partial)$. Given $H$, the homologies for different choices of generic $J$ are canonically isomorphic to each other, so we can denote this homology simply by $HF(H)$. 

The construction of the above canonical isomorphisms is a special case of the following more general construction. Given two admissible Hamiltonians $H_1,H_2\in\Hstd$, write $H_1\le H_2$ if $H_1(\theta,x)\le H_2(\theta,x)$ for all $(\theta,x)\in S^1\times\widehat{X}$. In this situation, one defines a {\bf continuation morphism\/} $HF(H_1)\to HF(H_2)$ as follows; cf.\ \cite[Thm.\ 4.5]{gutt} and the references therein. Choose generic $J_1,J_2\in\mathcal{J}$ so that the chain complexes $CF(H_i,J_i)$ are defined for $i=1,2$. Choose a generic homotopy $\{(H_s,J_s)\}_{s\in\R}$ such that $H_s$ satisfies equation \eqref{eqn:limitingslope} for some $\beta,\beta'$ depending on $s$; $J_s\in\mathcal{J}$ for each $s\in\R$; $\partial_sH_s\ge 0$; $(H_s,J_s)=(H_1,J_1)$ for $s<<0$; and $(H_s,J_s)=(H_2,J_2)$ for $s>>0$. One then defines a chain map $CF(H_1,J_1)\to CF(H_2,J_2)$ as a signed count of Fredholm index $0$ maps 
$u:\R\times S^1\rightarrow \widehat{X}$ satisfying the equation
\begin{equation}
\label{eq:floerparam}
		\frac{\partial u}{\partial s} + J_s^{\theta}\circ u\Bigl(\frac{\partial u}{\partial \theta} - X^\theta_{H_s}\circ u\Bigr)=0
\end{equation}
and the asymptotic conditions $\lim_{s\to-\infty}u(s,\cdot) = \gamma_1$ and $\lim_{s\to\infty}u(s,\cdot) = \gamma_2$. The induced map on homology gives a well-defined map $HF(H_1)\to HF(H_2)$. If $H_2\le H_3$, then the continuation map $HF(H_1)\to HF(H_3)$ is the composition of the continuation maps $HF(H_1)\to HF(H_2)$ and $HF(H_2)\to HF(H_3)$. 

\begin{definition}
We define the {\bf symplectic homology\/} of $(X,\lambda)$ to be the direct limit 
\[
SH(X,\lambda) := \lim_{\substack{\longrightarrow \\ H\in \mathcal{H}_{\textrm{std}}}}HF(H)
\]
with respect to the partial order $\le$ and continuation maps defined above.
\end{definition}

\subsection{Positive symplectic homology}
\label{sec:psh}

Positive symplectic homology is a modification of symplectic homology in which constant $1$-periodic orbits are discarded.

To explain this, let $H:S^1\times\widehat{X}\rightarrow \R$ be a Hamiltonian in $\Hstd$. The {\bf Hamiltonian action\/} functional $\mathcal{A}_{H}: C^{\infty}(S^{1},\widehat{X})\rightarrow \R$ is defined by
\[
	\mathcal{A}_{H}(\gamma) := -\int_{S^{1}}\gamma^{\star}{\widehat{\lambda}} - \int_{S^{1}}H\bigl(\theta,\gamma(\theta)\bigr)d\theta.
\]
If $J\in\mathcal{J}$, then the differential on the chain complex $(CF(H,J),\partial)$ decreases the Hamiltonian action $\mathcal{A}_H$. As a result, for any $L\in\R$, we have a subcomplex $CF^{\le L}(H,J)$ of $CF(H,J)$, generated by the $1$-periodic orbits with Hamiltonian action less than or equal to $L$. 

To see what this subcomplex can look like, note that the $1$-periodic orbits of $H \in \Hstd$ fall into two classes: (i) constant orbits corresponding to critical points in $X$, and (ii) non-constant orbits contained in $[0,\rho_0]\times Y$.

If $x$ is a critical point of $H$ on $X$, then the action of the corresponding constant orbit is equal to $-H(x)$. By \eqref{eqn:HSpec}, this is less than $\epsilon$.

By Remark~\ref{rem:orbitsofHstand}, a non-constant $1$-periodic orbit of $X_H$ is close to a $1$-periodic orbit of $-h'(e^\rho)R_{\lambda}$ located in $\{ \rho \}\times Y$ for  $\rho\in[0,\rho_0]$ with $h'(e^\rho)\in\Spec(Y,\lambda)$. The Hamiltonian action of the latter loop is given by 
\begin{equation}
\label{eqn:actionsclose}
-\int_{S^1}e^\rho\lambda(-h'(e^\rho)R_{\lambda})d\theta -\int_{S^1}h(e^\rho)d\theta=e^\rho h'(e^\rho)-h(e^\rho).
\end{equation}
Since $h$ is strictly convex, the right hand side is a strictly increasing function of $\rho$.

\begin{remark}
\label{rem:hsmall}
In Definition~\ref{def:Hstd}, we assume that $h$ is sufficiently small so that the right hand side of \eqref{eqn:actionsclose} is close to the period $h'(e^\rho)$, and in particular greater than $\epsilon$. We also assume that $H$ is sufficiently close to $h(e^\rho)$ on $S^1\times[0,\rho_0]\times Y$ so that the Hamiltonian actions of the $1$-periodic orbits are well approximated by the right hand side of \eqref{eqn:actionsclose}, so that:
\begin{description}
\item{(i)} The Hamiltonian action of every 1-periodic orbit of $X_H$ corresponding to a critical point on $X$ is less than $\epsilon$; and the Hamiltonian action of every other $1$-periodic orbit is greater than $\epsilon$.
\item{(ii)}
If $\gamma$ is a Reeb orbit of period $T<\beta$, and if $\gamma'$ is a $1$-periodic orbit of $X_H$ in $[0,\rho_0]\times Y$ associated to $\gamma$, then
\[
|\mc{A}_H(\gamma_1') - T | < \min\left\{\beta^{-1},\tfrac{1}{3}\op{gap}(\beta)\right\}.
\]
Here $\op{gap}(\beta)$ denotes the minimum difference between two elements of $\Spec(Y,\lambda)$ that are less than $\beta$. 
\end{description}
\end{remark}

We can now define positive symplectic homology.

\begin{definition}
\label{def:psh}
	Let $(X,\lambda)$ be a Liouville domain, let $H$ be a Hamiltonian in $\Hstd$, and let $J\in\mathcal{J}$.

Consider the quotient complex
\[
CF^+(H,J) \eqdef \frac{CF(H,J)}{CF^{\le\epsilon}(H,J)}.
\]
The homology of the quotient complex is independent of $J$, so we can denote this homology by $HF^+(H)$. More generally, if $H_1\le H_2$, then the chain map used to define the continuation map $HF(H_1)\to HF(H_2)$ descends to the quotient, since the Hamiltonian action is nonincreasing along a solution of \eqref{eq:floerparam} when the homotopy is nondecreasing. Thus we obtain a well-defined continuation map $HF^+(H_1)\to HF^+(H_2)$ satisfying the same properties as before.

We now define the {\bf positive symplectic homology\/} of $(X,\lambda)$ to be the direct limit
\[
SH^{+}(X,\lambda) := \lim_{{\substack{\longrightarrow\\ {H\in \Hstd}}}} HF^{+}(H).
\]
\end{definition}

Positive symplectic homology can sometimes be better understood using certain special admissible Hamiltonians obtained as follows.

\begin{definition}
\label{sec:perturbationMB}
\label{def:HMB} 
\cite{BOduke}
Let $(X,\lambda)$ be  a Liouville domain. An {\bf admissible Morse-Bott Hamiltonian\/} is an autonomous Hamiltonian $H: \widehat{X}\rightarrow \R$ such that:
\begin{description}
\item{(1)}
The restriction of $H$ to $X$ is a Morse function which is negative and $C^{2}$-small (so that the Hamiltonian vector field has no non-constant $1$-periodic orbits).
\item{(2)}
There exists $\rho_{0} \geq 0$
such that on $[\rho_0,\infty)\times Y$ we have
\[
H(\rho,x) = \beta e^{\rho} + \beta'
\]
with $0<\beta\notin\Spec(Y,\lambda)$  and $\beta'\in\R$.
\item{(3)}
On $[0,\rho_0)\times Y$ we have
\[
H(\rho,x) = h(e^{\rho})
\]
where $h$ is as in Definition~\ref{def:Hstd}, and moreover $h''-h'>0$.
\end{description}
We denote the set of admissible Morse-Bott Hamiltonians by $\mathcal{H}_{\textrm{MB}}$.
\end{definition}

Given $H\in\mathcal{H}_{\textrm{MB}}$, each $1$-periodic orbit of $X_{H}$ is either: (i) a constant orbit corresponding to a critical point of $H$ in $X$, or (ii) a non-constant $1$-periodic orbit, with image in $\{\rho\}\times Y$ for $\rho\in (0,\rho_{0})$, whose projection to $Y$ has the same image as a Reeb orbit of period $e^\rho h'(\rho)$. Since $H$ is autonomous, every Reeb orbit $\gamma$ with period less than $\beta$ gives rise to an $S^{1}$ family of $1$-periodic orbits of $X_{H}$, which we denote by $S_{\gamma}$.

An admissible Morse-Bott Hamiltonian as in Definition~\ref{def:HMB} can be deformed into an admissible Hamiltonian as in Definition~\ref{def:Hstd}, which will be time-dependent and have nondegenerate 1-periodic orbits:
	
\begin{lemma}
\label{lem:orbitesMB}
(\cite[Prop.\ 2.2]{CFHW} and \cite[Lem.\ 3.4]{BOduke})
An admissible Morse-Bott Hamiltonian $H$ can be perturbed to an admissible Hamiltonian $H'$ whose $1$-periodic orbits consist of the following:
\begin{description}
\item{(i)}
Constant orbits at the critical points of $H$.
\item{(ii)}
For each Reeb orbit $\gamma$ with period less than $\beta$, two nondegenerate orbits $\widehat{\gamma}$ and $\widecheck{\gamma}$.
Given a trivialization $\tau$ of $\xi|\gamma$, their Conley-Zehnder indices are given by
$-\CZ_\tau(\widehat{\gamma}) = \CZ_\tau(\gamma)+1$ and $-\CZ_\tau(\widecheck{\gamma}) = \CZ_\tau(\gamma)$.
\end{description}
\end{lemma}

\begin{remark}
The references \cite{CFHW} and \cite {BOduke} use the notation $\gamma_{\m}$ instead of $\widehat{\gamma}$, and $\gamma_{\Max}$ instead of $\widecheck{\gamma}$. The motivation is that these orbits are distinguished in their $S^1$-family as critical points of a perfect Morse function on $S^1$.
\end{remark}

\subsection{$S^1$-equivariant symplectic homology}
\label{SHS1}

Let $(X,\lambda)$ be a Liouville domain with boundary $Y$.
We now review how to define the $S^1$-equivariant symplectic homology $SH^{S^1}(X,\lambda)$, and the positive $S^1$-equivariant symplectic homology $SH^{S^1,+}(X,\lambda)$.

The $S^1$-equivariant symplectic homology $SH^{S^1}(X,\lambda)$ is defined as a limit as $N\to\infty$ of homologies $SH^{S^1,N}(X,\lambda)$, where $N$ is a nonnegative integer. To define the latter, fix the perfect Morse function $f_N:\C P^N\to\R$ defined by
\[
f_N\bigl([w^0:\ldots:w^n]\bigr)=\frac{\sum_{j=0}^Nj|w^j|^2}{\sum_{j=0}^N|w^j|^2}.
\]
Let $\widetilde{f}_N:S^{2N+1}\to\R$ denote the pullback of $f_N$ to $S^{2N+1}$. We will consider gradient flow lines of $\widetilde{f_N}$ and $f_N$ with respect to the standard metric on $S^{2N+1}$ and the metric that this induces on $\C P^N$.

\begin{remark}\label{rmk:periodicity}
The family of functions $f_N$ has the following two properties which are needed below. We have two isometric inclusions $i_0,i_1:\C P^N\to\C P^{N+1}$ defined by $i_0([z_0:\ldots:z_{N}]) = [z_0:\ldots:z_{N}:0]$ and $i_1([z_0:\ldots:z_{N}])= [0:z_0:\ldots:z_{N}]$. Then:
\begin{description}
\item{(1)}
The images of $i_0$ and $i_1$ are invariant under the gradient flow of $f_{N+1}$.
\item{(2)}
We have $f_N=f_{N+1}\circ i_0=f_{N+1}\circ i_1 + \op{constant}$, so that the gradient flow of $f_{N+1}$ pulls back via $i_0$ or $i_1$ to the gradient flow of $f_N$.
\end{description}
\end{remark}

Now choose a ``parametrized Hamiltonian''
\begin{equation}
\label{eqn:Hamparam}
H:S^1\times \widehat{X}\times S^{2N+1}\longrightarrow \R
\end{equation}
which is $S^1$-invariant in the sense that
\[ 
H(\theta+\varphi,x,\varphi z) = H(\theta,x,z)\qquad \forall \theta,\varphi\in S^1=\R/\Z,\; x\in\widehat{X},\; z\in S^{2N+1}.
\]
Here the action of $S^1=\R/\Z$ on $S^{2N+1}\subset \C^{N+1}$ is defined by $\varphi\cdot z=e^{2\pi i\varphi}z$.

\begin{definition}
\label{def:aph}
A parametrized Hamiltonian $H$ as above is {\bf admissible\/} if:
\begin{description}
\item{(i)}
For each $z\in S^{2N+1}$, the Hamiltonian 
\[
H_z = H(\cdot,\cdot,z):S^1\times\widehat{X}\longrightarrow\R
\]
satisfies conditions (1), (2), and (3) in Definition~\ref{def:Hstd}, with $\beta$ and $\beta'$ independent of $z$.
\item{(ii)}
If $z$ is a critical point of $\widetilde{f}_N$, then the $1$-periodic orbits of $H_z$ are nondegenerate.
\item{(iii)} $H$ is nondecreasing along downward gradient flow lines of $\widetilde{f}_N$.
\end{description}
\end{definition}

Let $\Per^{S^1}(\tilde{f}_N,H)$ denote the set of pairs $(z,\gamma)$, where $z\in S^{2N+1}$ is a critical point of $\tilde{f}_N$, and $\gamma$ is a $1$-periodic orbit of the Hamitonian $H_z$.
Note that $S^1$ acts freely on the set $\Per^{S^1}(\tilde{f}_N,H)$ by
\[
\varphi\cdot (z,\gamma) = \big(\varphi\cdot z, \gamma(\cdot - \varphi)\big).
\]
If $p=(z,\gamma)\in\Per^{S^1}(\tilde{f}_N,H)$, let $S_p$ denote the orbit of $(z,\gamma)$ under this $S^1$ action.

Next, choose a generic map
\begin{equation}
\label{eqn:Jparam}
J:S^1\times S^{2N+1} \to \mc{J}, \quad (\theta,z) \mapsto J^\theta_z,
\end{equation}
which is $S^1$-invariant in the sense that
\[
J^{\theta+\varphi}_{\varphi\cdot z} = J^\theta_z
\]
for all $\varphi,\theta\in S^1$ and $z\in S^{2N+1}$. 

Let $p^-=(z^-,\gamma^-)$ and $p^+=(z^+,\gamma^+)$ be distinct elements of $\Per^{S^1}(\tilde{f}_N,H)$. Define $\widehat{\Mod}(S_{p^-},S_{p^+};J)$
to be the set of pairs $(\eta,u)$, where $\eta:\R\to S^{2N+1}$ and $u:\R\times S^1\to\widehat{X}$, satisfying the following equations:
\begin{equation}
\label{eq:fleorparam}
\left\{
\begin{aligned}
\dot{\eta}+\vec{\nabla}\tilde{f}_N(\eta) &=0,\\
\partial_su+J^{\theta}_{\eta(s)}\circ u\bigl(\partial_{\theta}u-X_{H_{\eta(s)}^\theta}\circ u\bigr) &=0,\\
\lim_{s\rightarrow\pm\infty}\bigl(\eta(s),u(s,\cdot)\bigr) &\in S_{p^\pm}.
\end{aligned}
\right.
\end{equation}
Here the middle equation is a modification of Floer's equation \eqref{eqn:Floer} which is ``parametrized by $\eta$''. Note that $\R$ acts on the set $\widehat{\Mod}(S_{p^-},S_{p^+};J)$ by reparametrization: if $\sigma\in\R$, then
\[
\sigma\cdot(\eta,u) = \big(\eta(\cdot-\sigma),u(\cdot-\sigma,\cdot)\big).
\]
In addition, $S^1$ acts on the set $\widehat{\Mod}(S_{p^-}, S_{p^+};J)$ as follows: if $\tau\in S^1$, then
\[
\tau\cdot (\eta,u):=\bigl(\tau\cdot\eta, u(\cdot,\cdot-\tau)\bigr).
\]
Let $\Mod^{S^1}(S_{p^-},S_{p^+};J)$ denote the quotient of the set $\widehat{\Mod}(S_{p^-},S_{p^+};J)$ by these actions of $\R$ and $S^1$.

If $J$ is generic, then $\Mod^{S^1}(S_{p^-},S_{p^+};J)$ is a manifold near $(\eta,u)$ of dimension
\[
\op{ind}(\eta,u) = (\op{ind}(f_N,z^-) - \op{CZ}_\tau(\gamma^-)) - (\op{ind}(f_N,z^+) - \op{CZ}_\tau(\gamma^+)) - 1.
\]
Here $\op{ind}(f_N,z^\pm)$ denotes the Morse index of the critical point $z^\pm$ of $f_N$, and $\op{CZ}_\tau$ denotes the Conley-Zehnder index with respect to a trivialization $\tau$ of ${(\gamma^\pm)}^{\star}T\widehat{X}$ that extends over $u^{\star}T\widehat{X}$.

\begin{definition}
\label{def:complex}
\cite[\S2.2]{bo}
Define a chain complex $\left(CF^{S^1,N}(H,J),\partial^{S^1}\right)$ as follows. The chain module $CF^{S^1,N}(H,J)$ is the free $\Q$ module\footnote{It is also possible to define $SH^{S^1,+}$, using $\Z$ coefficients, as with $SH$.} generated by the orbits $S_p$.
If $S_{p^-},S_{p^+}$ are two such orbits, then the coefficient of $S_{p^+}$ in $\partial^{S^1}S_{p^-}$ is a signed count of elements $(\eta,u)$ of $\Mod^{S^1}(S_{p^-},S_{p^+};J)$ with $\op{ind}(\eta,u)=1$.
\end{definition}

We denote the homology of this chain complex by $HF^{S^1,N}(H)$. This does not depend on the choice of $J$, by the usual continuation argument; one defines continuation chain maps using a modification of \eqref{eq:fleorparam} in which the second line is replaced by an ``$\eta$-parametrized'' version of Floer's continuation equation \eqref{eq:floerparam}.

We now define a partial order on the set of pairs $(N,H)$, where $N$ is a nonnegative integer and $H$ is an admissible parametrized Hamiltonian \eqref{eqn:Hamparam}, as follows. Let $\widetilde{i}_0:S^{2N+1}\to S^{2N+3}$ denote the inclusion sending $z\mapsto (z,0)$. (This lifts the inclusion $i_0$ defined in Remark~\ref{rmk:periodicity}.) Then $(N_1,H_1)\le (N_2,H_2)$ if and only if:
\begin{itemize}
\item
$N_1\le N_2$, and
\item
$H_1 \le (\widetilde{i_0}^{\star})^{N_2-N_1}H_2$ pointwise on $S^1\times\widehat{X}\times S^{2N_1+1}$.
\end{itemize}
In this case we can define a continuation map $HF^{S^1,N_1}(H_1) \to HF^{S^2,N_2}(H_2)$ using an increasing homotopy from $H_1$ to $(\widetilde{i_0}^{\star})^{N_2-N_1}H_2$ on $S^1\times\widehat{X}\times S^{2N_1+1}$.

\begin{definition}
Define the {\bf $S^1$-equivariant symplectic homology\/}
\[
SH^{S^1}_*(X,\lambda) := \lim_{\substack{\longrightarrow\\ N,H}}HF^{S^1,N}_*(H).
\]
\end{definition}

It is sometimes useful to describe $S^1$-equivariant symplectic homology in terms of individual Hamiltonians on $S^1\times \widehat{X}$, rather than $S^{2N+1}$-families of them, by the following procedure.

\begin{remark}
\label{rmk:simplifyingcomplex}
\cite[\S2.1.1]{these}
Fix an admissible Hamiltonian $H': S^1\times \widehat{X}\rightarrow \R$ and a nonnegative integer $N$. Consider a sequence of admissible parametrized Hamiltonians $\{H_k\}_{k=0,\ldots,N}$ as in \eqref{eqn:Hamparam}, where $H_k$ is defined on $S^1\times\widehat{X}\times S^{2k+1}$, with the following properties:
\begin{itemize}
\item
For each $k=0,\ldots,N-1$, the pullbacks $\widetilde{i}_0^{\star}H_{k+1}$ and $\widetilde{i}_1^{\star}H_{k+1}$ agree with $H_k$ up to a constant. Here $\widetilde{i}_1:S^{2k+1}\to S^{2k+3}$ denotes the lift of $i_1$ sending $z\mapsto (0,z)$.
\item
For each $k=0,\ldots,N$ and each $z\in\op{Crit}(\tilde{f}_k)$, we have
\begin{equation}
\label{eqn:HNCrit}
H_k(\theta,x,z)=H'\big(\theta-\phi(z),x\big) + c.
\end{equation}
Here $c$ is a constant depending on $k$ and $z$; and the map $\phi:\op{Crit}(\tilde{f}_k)\to S^1$ sends a critical point $(0,\ldots,0,e^{2\pi i \psi},0,\ldots,0)\mapsto \psi$.
\end{itemize}
Next, choose a sequence of families of almost complex structures $J_k:S^1\times S^{2k+1} \to \mc{J}(\widehat{X})$ for $k=0,\ldots,N$ such that:
\begin{itemize}
\item $J_k$ is generic so that the chain complex $\left(CF^{S^1,k}(H_k,J_k),\partial^{S^1}\right)$ is defined.
\item
$\widetilde{i}_0^{\star}J_{k+1} = \widetilde{i}_1^{\star}J_{k+1} = J_k$.
\end{itemize}

The chain complex $\left(CF^{S^1,N}(H_N,J_N),\partial^{S^1}\right)$ can now be described as follows.
By \eqref{eqn:HNCrit}, we can identify the chain module as
\begin{equation}
\label{eqn:cmsimplified}
CF^{S^1,N}(H_N,J_N) = \Q\{1,u,\ldots,u^N\}\otimes_\Q CF(H',J_0).
\end{equation}
This identification sends a pair $(z,\gamma)$, where $z\in\op{Crit}(\widetilde{f}_N)$ is a lift of an index $2k$ critical point of $f_N$ and $\gamma$ is a reparametrization of a $1$-periodic orbit $\gamma'$ of $H'$, to $u^k\tensor\gamma'$.

Since the sequences $\{H_k\}$ and $\{J_k\}$ respect the inclusions $\widetilde{i}_1$, the differential has the form
\begin{equation}
\label{eqn:partialsimplified}
\partial^{S^1}(u^k\tensor\gamma) = \sum_{i=0}^ku^{k-i}\tensor\varphi_i(\gamma)
\end{equation}
where the operator $\varphi_i$ on $CF(H',J_0)$ does not depend on $k$. In particular, $\varphi_0$ is the differential on $CF(H',J_0)$. We can also formally write
\[
\partial^{S^1} = \sum_{i=0}^Nu^{-i}\tensor\varphi_i
\]
where it is understood that $u^{-i}$ annihilates terms of the form $u^j\tensor\gamma$ with $i>j$.

The usual continuation arguments show that the homology of this chain complex does not depend on the choice of sequences $\{H_k\}$ and $\{J_k\}$ satisfying the above assumptions. We denote this homology by $HF^{S^1,N}(H')$.

Since in the above construction we assume that the sequences $\{H_k\}$ and $\{J_k\}$ respect the inclusions $\widetilde{i}_0$, it follows that when $N_1\le N_2$ we have a well-defined map $HF^{S^1,N_1}(H')\to HF^{S^1,N_2}(H')$ induced by inclusion of chain complexes.

As before, if $H_1'\le H_2'$, then there is a continuation map $HF^{S^1,N}(H_1')\to HF^{S^1,N}(H_2')$ satisfying the usual properties.
\end{remark}

As in \cite[\S2.3]{bo}, we now have:

\begin{proposition}
\label{shortS1}
The $S^1$-equivariant homology of $(X,\lambda)$ is given by
\[
SH_*^{S^1}(X,\lambda) = \lim_{\substack{\longrightarrow \\ N\in\N,\; H'\in\Hstd}} HF^{S^1,N}(H').
\]
\end{proposition}

\subsection{Positive $S^1$-equivariant symplectic homology}
\label{sec:SHS1+}

Like symplectic homology, $S^1$-equivariant symplectic homology also has a positive version in which constant $1$-periodic orbits are discarded.

\begin{definition}
Let $H:S^1\times\widehat{X}\times S^{2N+1}\to\R$ be an admissible parametrized Hamiltonian. The {\bf parametrized action functional\/} $\mc{A}_H:C^{\infty}(S^1,\widehat{X})\times S^{2N+1}\longrightarrow\R$  is defined by
\begin{equation}
\label{eq:paramaction}	\mc{A}_H(\gamma,z):=-\int_{\gamma}\widehat{\lambda}-\int_{S^1}H\bigl(\theta,\gamma(\theta),z\bigr)d\theta.
\end{equation}
\end{definition}

\begin{lemma}
\label{lem:dpa}
If $H$ is an admissible parametrized Hamiltonian, and if $J$ is a generic $S^1$-invariant family of almost complex structures as in \eqref{eqn:Jparam}, then the differential $\partial^{S^1}$ on $CF^{S^1,N}(H,J)$ does not increase the parametrized action \eqref{eq:paramaction}.
\end{lemma}

\begin{proof}
Given a solution $(\eta,u)$ to the equations \eqref{eq:fleorparam}, one can think of $\eta$ as fixed and regard $u$ as a solution to an instance of equation \eqref{eq:floerparam}, where $J_s$ and $H_s$ in \eqref{eq:floerparam} are determined by $\eta$. By condition (iii) in Definition~\ref{def:aph}, this instance of \eqref{eq:floerparam} corresponds to a nondecreasing homotopy of Hamiltonians. Consequently, the action is nonincreasing along this solution of \eqref{eq:floerparam} as before.
\end{proof}

It follows from Lemma~\ref{lem:dpa} that for any $L\in\R$, we have a subcomplex $CF^{S^1,N,\le L}(H,J)$ of $CF^{S^1,N}(H,J)$, spanned by $S^1$-orbits of pairs $(z,\gamma)$ where $z\in\op{Crit}(\tilde{f}_N)$ and $\gamma$ is a $1$-periodic orbit of $H_z$ with $\mc{A}_H(z,\gamma)\le L$.

As in \S\ref{sec:psh}, if the $S^1$-orbit of $(z,\gamma)$ is a generator of $CF^{S^1,N}(H,J)$, then there are two possibilities: (i) $\gamma$ is a constant orbit corresponding to a critical point of $H_z$ on $X$, and $\mc{A}_H(z,\gamma)<\epsilon$; or (ii) $\gamma$ is close to a Reeb orbit in $\{\rho\}\times Y$ with period $-h'(e^\rho)$, and $\mc{A}_H(z,\gamma)$ is close to this period; in particular $\mc{A}_H(z,\gamma)>\epsilon$.

\begin{definition}
\label{def:positiveequivSH}
Consider the quotient complex
\begin{equation}
\label{eqn:peshquotient}
CF^{S^1,N,+}(H,J) := \frac{CF^{S^1,N}(H,J)}{CF^{S^1,N,\leq\epsilon}(H,J)}.
\end{equation}
As in Definition~\ref{def:psh}, the homology of the quotient complex is independent of $J$, so we can denote this homology by $HF^{S^1,N,+}(H)$; and we have continuation maps $HF^{S^1,N_1,+}(H_1)\to HF^{S^1,N_2,+}(H_2)$ when $(N_1,H_1)\le (N_2,H_2)$. We now define the {\bf positive $S^1$-equivariant symplectic homology\/} by
\begin{equation}
\label{eqn:pesh}
SH^{S^1,+}(X,\lambda) := \lim_{\substack{\longrightarrow\\ {N,H}}}HF^{S^1,N,+}(H).
\end{equation}
\end{definition}

Returning to the situation of Remark \ref{rmk:simplifyingcomplex}, define $HF^{S^1,N,+}(H')$ to be the homology of the quotient of the chain complex \eqref{eqn:cmsimplified} by the subcomplex spanned by $u^k\tensor\gamma$ where $\gamma$ is a critical point of $H'$ in $X$. We then have the following analogue of Proposition \ref{shortS1}:

\begin{proposition}
\label{shortS1+}
The positive $S^1$ equivariant homology of $(X,\lambda)$ is given by
\[
SH^{S^1,+}(X,\lambda) = \lim_{\substack{\longrightarrow \\ N\in\N,\; H'\in\Hstd}} HF^{S^1,N,+}(H').
\]
\end{proposition}

%% file: eh-proofs1.tex
\section{Properties of positive $S^1$-equivariant SH}
\label{sec:proofs1}

Let $(X,\lambda)$ be a Liouville domain.  We now show that the positive $S^1$-equivariant homology $SH^{S^1,+}(X,\lambda)$ defined in \S\ref{sec:equivariantsh}, which we denote by $CH(X,\lambda)$ for short, satisfies all of the properties in Proposition~\ref{prop:ch}.

\subsection{Free homotopy classes}

Given an admissible Hamiltonan, $H$, we can decompose the complex $CF^{S^1,N,\leq L}(H,J)$ into a direct sum
\[
CF^{S^1,N,\leq L}(H,J) = \bigoplus_{\Gamma}CF^{S^1,N,\leq L}(H,J,\Gamma).
\]
Here $\Gamma$ ranges over free homotopy classes of loops in $X$, and $CF^{S^1,N,\leq L}(H,J,\Gamma)$ denotes the subset of $CF^{S^1,N,\le L}(H,J)$ generated by $S^1$-orbits of pairs $(z,\gamma)$ where $\gamma$ represents the free homotopy class $\Gamma$.

The differentials and continuation maps defined in \S\ref{sec:equivariantsh} all count certain cylinders, and thus respect the above direct sum decomposition. As a result, we obtain a corresponding direct sum decomposition in \eqref{eqn:peshquotient} and \eqref{eqn:pesh}, so that we can decompose
\[
CH(X,\lambda) = \bigoplus_{\Gamma}CH(X,\lambda,\Gamma),
\]
where $CH(X,\lambda,\Gamma)$ is defined like $CH(X,\lambda)$ but only using loops in the free homotopy class $\Gamma$.

Similar remarks apply to all of the constructions to follow; we will omit the free homotopy class $\Gamma$ below to simplify notation.

\subsection{Action filtration}
\label{pf:actionfiltration}

Given $L\in\R$, we now define a version of positive $S^1$-equivariant symplectic homology ``filtered up to action $L$'', which we denote by $CH^L(X,\lambda)$. This will only depend on the largest element of $\Spec(Y,\lambda)$ which is less than or equal to $L$. Thus we can assume without loss of generality that $L\notin\Spec(Y,\lambda)$.

As in Definition \ref{def:positiveequivSH}, we can consider the quotient complex
\[
CF^{S^1,N,+,\leq L}(H,J) := \frac{CF^{S^1,N,\leq L}(H,J)}{CF^{S^1,N,\leq\epsilon}(H,J)}.
\]
As in Definition~\ref{def:psh}, the homology of the quotient complex is independent of $J$, so we can denote this homology by $HF^{S^1,N,+,\leq L}(H)$. If $(N_1,H_1)\le (N_2,H_2)$, then the continuation chain map induces a well-defined map $HF^{S^1,N_1,+,\leq L}(H_1)\to HF^{S^1,N_2,+,\leq L}(H_2)$.

\begin{definition}
We define the positive $S^1$-equivariant symplectic homology filtered up to action $L$ to be
\[
CH^L(X,\lambda)\eqdef SH^{S^1,+,\leq L}(X,\lambda) \eqdef \lim_{\substack{\longrightarrow\\ {N,H}}}HF^{S^1,N,+,\leq L}(H).
\]
\end{definition}
It follows from Remark~\ref{rem:hsmall}(ii) that if $L\notin\Spec(Y,\lambda)$, then $CH^L(X,\lambda)$ depends only on the largest element of $\Spec(Y,\lambda)$ that is less than $L$.

Given an admissible parametrized Hamiltonian $H$, a nonnegative integer $N$, a generic parametrized almost complex structure $J$ as in \eqref{eqn:Jparam}, and real numbers $L_1<L_2$, we have an inclusion of chain complexes
\begin{equation}
\label{eqn:icc}
CF^{S^1,N,+,\le L_1}(H,J) \longrightarrow CF^{S^1,N,+,\le L_2}(H,J).
\end{equation}
The usual continuation map argument shows that the induced map on homology,
\begin{equation}
\label{eqn:icch}
HF^{S^1,N,+,\leq L_1}(H)\longrightarrow HF^{S^1,N,+,\leq L_2}(H),
\end{equation}
does not depend on the choice of $J$, and commutes with the continuation map for $(N_1,H_1) \le (N_2,J_2)$.

\begin{definition}
\label{def:il2l1}
We define the map
\begin{equation}
\label{eqn:imathl2l1}
\imath_{L_2,L_1} : CH^{L_1}(X,\lambda) \longrightarrow CH^{L_2}(X,\lambda)
\end{equation}
to be the direct limit over pairs $(N,H)$ of the maps \eqref{eqn:icch}.
\end{definition}

We then have the required property
\begin{equation}
\label{eqn:tdl}
\lim_{L\to \infty} CH^L(X,\lambda) = CH(X,\lambda),
\end{equation}
because we can compute the direct limit
\[
\lim_{\substack{\longrightarrow\\{N,H,L}}} HF^{S^1,N,+,\le L}(H)
\]
either by first taking the limit over pairs $(N,H)$, which gives the left hand side of \eqref{eqn:tdl}, or by first taking the limit over $L$, which gives the right hand side of \eqref{eqn:tdl}.

\begin{remark}
One can equivalently define $CH^L(X,\lambda)$ by repeating the definition of $CH(X,\lambda)$, but using appropriate admissible Hamiltonians where the limiting slope is equal to $L$.
\end{remark}

\subsection{$U$ map}
\label{sec:U}

We now define the $U$ map on $CH(Y,\lambda)$, similarly to \cite[\S2.4]{bo}.

Recall from Remark~\ref{rmk:simplifyingcomplex} that given an admissible Hamiltonian $H':S^1\times\widehat{X}\to\R$ and a nonnegative integer $N$, we can choose a pair $(H_N,J_N)$ so that the chain complex $\left(CH^{S^1,N}(H_N,J_N),\partial^{S^1}\right)$ has the nice form given by \eqref{eqn:cmsimplified} and \eqref{eqn:partialsimplified}.

It follows from \eqref{eqn:partialsimplified} that the operation of ``multiplication by $u^{-1}$'', sending a chain complex generator $u^i\tensor\gamma$ to $u^{i-1}\tensor\gamma$ when $i>0$ and to $0$ when $i=0$, is a chain map. This induces a map on the homology $HF^{S^1,N}(H')$, which we denote by $U_{N,H'}$. A priori this map also depends on the choice of pair $(H_N,J_N)$, but the usual continuation map argument shows that it does not. In addition, if $(N_1,H_1')\le (N_2,H_2')$, then the continuation map $HF^{S^1,N_1}(H_1')\to HF^{S^1,N_2}(H_2')$ fits into a commutative diagram
\[
\begin{CD}
HF^{S^1,N_1}(H_1') @>>> HF^{S^1,N_2}(H_2')\\
@V{U_{N_1,H_1'}}VV @VV{U_{N_2,H_2'}}V\\
HF^{S^1,N_1}(H_1') @>>> HF^{S^1,N_2}(H_2').
\end{CD}
\]
It then follows from Proposition~\ref{shortS1} that we obtain a well-defined map
\[
U = \lim_{\substack{\longrightarrow\\N,H'}}U_{N,H'}
\]
on $SH_*^{S^1}(X,\lambda)$.

Since the $U$ map is induced by chain maps which respect (in fact preserve) the symplectic action filtration, it also follows from Proposition~\ref{shortS1+} that we obtain a well-defined $U$ map on $CH(Y,\lambda)$. Similarly we obtain a well-defined $U$ map on $CH^L(Y,\lambda)$. This completes the proof of the ``$U$ map'' property.

For use in \S\ref{sec:starshaped} below, we also note that there is the following Gysin-type exact sequence:

\begin{proposition}
If $(X,\lambda)$ is a Liouville domain, then there is a long exact sequence
\begin{equation}
\label{Gysin}
\xymatrix{\cdots\ar[r]&SH^+(X,\lambda)\ar[r] & CH(X,\lambda)\ar[r]^U & CH(X,\lambda)\ar[r] & SH^{+}(X,\lambda)\ar[r] & \cdots}
\end{equation}
\end{proposition}

\begin{proof}
With the above definition of $U$, this follows as in \cite[Prop.\ 2.9]{bo}. This was also shown earlier in \cite{BOjta} using a slightly different definition of positive $S^1$-equivariant symplectic homology.
\end{proof}

\subsection{Reeb Orbits}

Let $L_1<L_2$ such that there does not exist a Reeb orbit $\gamma$ of $\lambda|_{\partial X}$ having action $\mc{A}(\gamma)$ in the interval $(L_1,L_2]$. As in \S\ref{pf:actionfiltration}, we can also assume without loss of generality that $L_1\notin\Spec(Y,\lambda)$. Then for every triple $(N,H,J)$, if the limiting slope of $H$ is sufficiently large, then the inclusion of chain complexes \eqref{eqn:icc} is the identity map. It follows that the map \eqref{eqn:icch} is an isomorphism, and consequently the direct limit map \eqref{eqn:imathl2l1} is an isomorphism as desired.

\subsection{$\delta$ map}
\label{sec:deltamap}

To define the delta map, we have the following:

\begin{proposition}
Let $(X,\lambda)$ be a Liouville domain. Then there is a canonical long exact sequence
\begin{equation}
\label{eq:LESquotient}
			\xymatrix{
			H_*(X,\partial X)\otimes H_*(BS^1) \ar[rr]   & & SH^{S^1}(X,\lambda)\ar[ld] \\
			& SH^{S^1,+}(X,\lambda) \ar[lu]^{\delta} &
		}.
\end{equation}
\end{proposition}

\begin{proof}
For any triple $(N,H,J)$ as in Definition~\ref{def:positiveequivSH}, by definition we have a short exact sequence of chain complexes
\begin{equation}
\label{eqn:sescc}
0 \longrightarrow CF^{S^1,N,\leq\epsilon}(H,J) \longrightarrow CF^{S^1,N}(H,J) \longrightarrow CF^{S^1,N,+}(H,J) \longrightarrow 0.
\end{equation}
Since continuation maps respect symplectic action, we can take the direct limit of the resulting long exact sequences on homology to obtain a canonical long exact sequence
\begin{equation}
\label{eqn:cles}
\cdots \longrightarrow
SH^{S^1,\le\epsilon}(X,\lambda) \longrightarrow SH^{S^1}(X,\lambda) \longrightarrow SH^{S^1,+}(X,\lambda) \longrightarrow \cdots
\end{equation}
where we define
\begin{equation}
\label{eqn:shmorse}
SH^{S^1,\le \epsilon}(X,\lambda) = \lim_{\substack{\longrightarrow\\N,H}} HF^{S^1,N,\le\epsilon}(H,J).
\end{equation}

To compute \eqref{eqn:shmorse}, note that we have a canonical isomorphism
\begin{equation}
\label{eqn:counterpart}
HF^{S^1,N,\le\epsilon}(H,J) = H_*(X,\partial X) \tensor \Q\{1,u,\ldots,u^N\}.
\end{equation}
For proofs of counterparts of this isomorphism for different definitions of $S^1$-equivariant symplectic homology, see \cite[Proposition 1.3]{V} and \cite[Lemma 4.8]{BOjta}. In our context, the isomorphism \eqref{eqn:counterpart} holds because if we compute the left hand side as in Remark~\ref{rmk:simplifyingcomplex}, then the chain complex comes from the critical points of $H'$ on $X$, so that we have
\begin{equation}
\label{eqn:Morse}
CF^{S^1,N,\le\epsilon}(H,J) = C_{\op{Morse}}(X,H')\tensor\Q\{1,u,\ldots,u^N\}.
\end{equation}
Here $C_{\op{Morse}}(X,H')$ denotes the chain complex for the Morse cohomology of $H'$, whose differential counts upward gradient flow lines; and $u^i$ represents the index $2i$ critical point of $f_N$. The differential on the left side of \eqref{eqn:Morse} agrees on the right side with the tensor product of the Morse differential and the identity on $\Q\{1,u,\ldots,u^N\}$. Since the gradient of $H'$ points out of $X$ along $\partial X$, the Morse cohomology agrees with the relative homology $H_*(X,\partial X)$. This proves \eqref{eqn:counterpart}, and taking the direct limit over pairs $(N,H)$ gives a canonical isomorphism
\begin{equation}
\label{eqn:SHcritical}
SH^{S^1,\le\epsilon}(X,\lambda) = H_*(X,\partial X)\tensor H_*(BS^1).
\end{equation}
Putting this into \eqref{eqn:cles} proves the proposition.
\end{proof}

The map $\delta$ vanishes on $CH(X,\lambda,\Gamma)$ for every free homotopy class $\Gamma\neq 0$, because the maps in the long exact sequence \eqref{eqn:cles} preserve the free homotopy class, and the homology \eqref{eqn:SHcritical} is entirely supported in the summand corresponding to $\Gamma=0$.

\subsection{Scaling}

If $(\widehat{X},\widehat{\lambda})$ is the completion of $(X,\lambda)$, then the completion of $(X,r\lambda)$ is naturally identified with the same manifold $\widehat{X}$, with the $1$-form $r\widehat{\lambda}$.

If $H:S^1\times \widehat{X}\to\R$ is an $S^1$-dependent Hamiltonian, and if $X_H$ denotes the ($S^1$-dependent) Hamiltonian vector field for $H$ defined using $\widehat{\omega}$, then the Hamiltonian vector field for $H$ defined using $r\widehat{\omega}$ is $r^{-1}X_H$. It follows that if $H$ is an admissible Hamiltonian for $(X,\lambda)$, then $rH$ is an admissible Hamiltonian for $(X,r\lambda)$, with the same $1$-periodic orbits. Note here that $\Spec(Y,r\lambda) = r\Spec(Y,\lambda)$, so the conditions involving the action spectrum are preserved. In particular, if $\epsilon=\frac{1}{2}\min\Spec(Y,\lambda)$ as usual, then
\[
r\epsilon = \frac{1}{2}\min\Spec(Y,r\lambda).
\]

Likewise, if $H:S^1\times \widehat{X}\times S^{2N+1}\to\R$ is an admissible parametrized Hamiltonian for $(X,\lambda)$, then $rH$ is an admissible parametrized Hamiltonian for $(X,r\lambda)$.

If $J$ is an admissible parametrized almost complex structure \eqref{eqn:Jparam} as needed to define the (positive) $S^1$-equivariant symplectic homology of $(X,\lambda)$, then $J$ is not quite admissible for $(X,r\lambda)$, because the condition \eqref{eqn:JReeb} only holds up to a constant. However one can still define (positive) $S^1$-equivariant symplectic homology using parametrized almost complex structures that satisfy this weaker version of admissibility, cf.\ \cite[\S1.3.2]{oanceasurvey}, and a continuation argument shows that the resulting (positive) $S^1$-equivariant symplectic homology will be canonically isomorphic.

Putting this together, we have a canonical isomorphism of chain complexes
\[
CF^{S^1,N,\le L}(H,J) = CF^{S^1,N,\le rL}(rH,J).
\]
We then have a canonical isomorphism of quotient chain complexes
\[
\frac{CF^{S^1,N,\le L}(H,J)}{CF^{S^1,N,\le \epsilon}(H,J)} = \frac{CF^{S^1,N,\le rL}(rH,J)}{CF^{S^1,N,\le r\epsilon}(rH,J)}.
\]
Taking the direct limit over pairs $(N,H)$ gives the desired canonical isomorphism
\[
CH^L(X,\lambda) = CH^{rL}(X,r\lambda).
\]
We can also take $L=+\infty$, giving the desired canonical isomorphism
\[
CH(X,\lambda) = CH(X,r\lambda).
\]

These scaling isomorphisms preserve the $U$ and $\delta$ maps since the holomorphic curves counted are the same.

\subsection{Star-Shaped Domains}
\label{sec:starshaped}

When $X$ is a nice star-shaped domain, the chain complex $CF^{S^1,N}(H,J)$ has a canonical $\Z$ grading, in which the grading of a pair $(z,\gamma)$ is $\op{ind}(z)-CZ(\gamma)$. Here $\op{ind}(z)$ denotes the Morse index of the corresponding critical point of $f_N$, while $CZ(\gamma)$ denotes the Conley-Zehnder index of $\gamma$, computed using a global trivialization of $TX$.

With respect to this grading, the long exact sequence \eqref{eq:LESquotient} has the form
\begin{equation}
			\xymatrix{
			H_{*+n}(X,\partial X)\otimes H_*(BS^1) \ar[rr]   & & SH^{S^1}_*(X)\ar[ld] \\
			& SH^{S^1,+}_*(X) \ar[lu]_{[-1]}^{\delta} &
		}.
\end{equation}
For a nice star-shaped domain $X$, we have $SH_*^{S^1}(X)=0$; see \cite[\S1.3.2]{these}. Assertions (i) and (ii) in the Star-Shaped Domains property follow. (The computation \eqref{eqn:sscomputation} also follows from \cite[Thm.\ 1.1]{gutt} together with the description of the Reeb orbits on the boundary of an ellipsoid in the proof of Lemma~\ref{lem:ellipsoid}.)

To prove assertion (iii), note that for a nice star-shaped domain, the Gysin-type sequence \eqref{Gysin} with gradings has the form
\[
\xymatrix{\cdots\ar[r]&SH^+_k(X)\ar[r] & CH_k(X)\ar[r]^U & CH_{k-2}(X)\ar[r] & SH^{+}_{k-1}(X)\ar[r] & \cdots}.
\]
On the other hand, if $X$ is a nice star-shaped domain then
\[
SH^+_*(X)=
\begin{cases}
		\Q&\textrm{if }*=n+1\\
		0&\textrm{otherwise}
\end{cases},
\]
see \cite[\S1.2.4]{these}. Therefore the $U$ map $CH_*(X,\lambda)\to CH_{*-2}(X,\lambda)$ is an isomorphism except when $*=n+1$.

Finally, we need to prove assertion (iv). Suppose that ${\lambda_0}|_{\partial X}$ is nondegenerate and has no Reeb orbit $\gamma$ with action $\mc{A}(\gamma)\in(L_1,L_2]$ and Conley-Zehnder index $\op{CZ}(\gamma) = n-1+2k$. We need to show that the map
\begin{equation}
\label{eqn:ssis}
\imath_{L_2,L_1}: CH^{L_1}_{n-1+2k}(X,\lambda_0) \longrightarrow CH^{L_2}_{n-1+2k}(X,\lambda_0)
\end{equation}
is surjective. As in \S\ref{pf:actionfiltration}, we can assume without loss of generality that $L_1,L_2\notin\Spec(Y,\lambda)$.

To prove that \eqref{eqn:ssis} is surjective, we compute positive $S^1$-equivariant symplectic homology using an admissible Hamiltonian $H':S^1\times\widehat{X}\to\R$ as in Remark~\ref{rmk:simplifyingcomplex}. Furthermore, we assume that $H'$ is perturbed from an admissible Morse-Bott Hamiltonian as in Lemma~\ref{lem:orbitesMB}, with boundary slope $\beta>L_2$. As a result, if $L<\beta$ is not close to the action of a Reeb orbit, then the chain complex $CF^{S^1,N,+,\le L}(H_N,J_N)$ is generated by symbols $u^k\tensor\widecheck{\gamma}$ and $u^k\tensor\widehat{\gamma}$ where $0\le k\le N$ and $\gamma$ is a Reeb orbit with action $\mc{A}(\gamma)\le L$. Furthermore, the grading of a generator is given by
\[
\begin{split}
|u^k\tensor\widecheck{\gamma}| &= \op{CZ}(\gamma) + 2k,\\
|u^k\tensor\widehat{\gamma}| &= \op{CZ}(\gamma) + 2k+1.
\end{split}
\]

Now fix $N$, $H_N$, and $J_N$. The differential on the chain complex $CF^{S^1,N,+,\le L}(H_N,J_N)$ does not increase the symplectic action of Reeb orbits. This means that we can define an integer-valued filtration $\mc{F}$ on the chain complex as follows: Denote the real numbers in the action spectrum $\Spec(Y,\lambda)$ by 
\[
a_1<a_2<\cdots.
\]
If $\gamma$ is a Reeb orbit with action $\mc{A}(\gamma)=a_j$, then we define the filtration
\[
\mc{F}(u^i\tensor\widecheck{\gamma}) = \mc{F}(u^i\tensor\widehat{\gamma}) = j.
\]
Let $\mc{F}_jCF^{\le L}$ denote the subcomplex of $CF^{S^1,N,+,\le L}(H_N,J_N)$ spanned by generators with filtration $\le j$. Let
\[
\mc{G}_jCF^{\le L} = \mc{F}_jCF^{\le L}/\mc{F}_{j-1}CF^{\le L}
\]
denote the associated graded complex.

It is shown in \cite[\S3.2]{gutt} that the homology of $\bigoplus_j\mc{G}_jCF^{\le L}$ is generated by $u^0\tensor\widecheck{\gamma}$ and $u^N\tensor\widehat{\gamma}$ where $\gamma$ ranges over the good Reeb orbits with action less than $L$. It follows that if $N$ is sufficiently large with respect to $k$ and $L$, then the grading $n-1+2k$ part of $\bigoplus_j\mc{G}_jCF^{\le L}$ is generated by $u^0\tensor\widecheck{\gamma}$ where $\gamma$ is a good Reeb orbit with action less than $L$ and Conley-Zehnder index equal to $n-1+2k$. Therefore, the inclusion of chain complexes
\begin{equation}
\label{eqn:ssinclusion}
CF^{S^1,N,+,\le L_1}(H_N,J_N) \longrightarrow CF^{S^1,N,+,\le L_2}(H_N,J_N)
\end{equation}
induces an injection
\[
\mc{G}_jCF^{\le L_1} \longrightarrow \mc{G}_jCF^{\le L_2}
\]
for each $j$. Furthermore, under our assumption on $k$, $L_1$, and $L_2$, if $N$ is sufficiently large, then the above injection in grading $n-1+2k$ is an isomorphism
\[
\mc{G}_jCF_{n-1+2k}^{\le L_1} \stackrel{\simeq}{\longrightarrow} \mc{G}_jCF_{n-1+2k}^{\le L_2}
\]
for each $j$. It now follows from the algebraic Lemma~\ref{lem:algebra} below that the inclusion \eqref{eqn:ssinclusion} induces a surjection on the degree $n-1+2k$ homology
\begin{equation}
\label{eqn:surjection}
HF_{n-1+2k}^{S^1,N,+,\le L_1}(H_N,J_N) {\longrightarrow} HF_{n-1+2k}^{S^1,N,+,\le L_2}(H_N,J_N).
\end{equation}

\begin{lemma}
\label{lem:algebra}
Let
\begin{gather*}
0=\mc{F}_0C_* \subset \mc{F}_1C_* \subset \cdots \subset \mc{F}_JC_*=C_*,
\\
0=\mc{F}_0C_*' \subset \mc{F}_1C_*' \subset \cdots \subset \mc{F}_JC_*'=C_*'
\end{gather*}
be filtered chain complexes. Denote the associated graded chain complexes by $\mc{G}_jC_*=\mc{F}_jC_*/\mc{F}_{j-1}C_*$ and $\mc{G}_jC_*'=\mc{F}_jC_*' / \mc{F}_{j-1}C_*'$. Let $\phi:C_*\to C_*'$ be a map of filtered chain complexes. For a given grading $k$, suppose that for each $j$, the map $\phi$ induces a surjection $H_k(\mc{G}_jC_*)\to H_k(\mc{G}_jC_*')$ and an injection $H_{k-1}(\mc{G}_jC_*)\to H_{k-1}(\mc{G}_jC_*')$. Then $\phi$ induces a surjection $H_kC_*\to H_kC_*'$ and an injection $H_{k-1}C_*\to H_{k-1}C_*'$.
\end{lemma}

\begin{proof}
Since the filtrations are bounded, it is enough to prove by induction on $j$ that $\phi$ induces a surjection $H_k(\mc{F}_jC_*)\to H_k(\mc{F}_jC_*')$ and an injection $H_{k-1}(\mc{F}_jC_*) \to H_{k-1}(\mc{F}_jC_*')$. Assume that the claim holds for $j-1$. We then have a commutative diagram with exact rows
\[
\begin{CD}
H_k(\mc{F}_{j-1}C_*) @>>> H_k(\mc{F}_jC_*) @>>> H_k(\mc{G}_jC_*) @>>> H_{k-1}(\mc{F}_{j-1}C_*) \\
@VV{\op{surj}}V @VVV @VV{\op{surj}}V @VV{\op{inj}}V\\
H_k(\mc{F}_{j-1}C_*') @>>> H_k(\mc{F}_jC_*') @>>> H_k(\mc{G}_jC_*') @>>> H_{k-1}(\mc{F}_{j-1}C_*')
\end{CD}
\]
where the vertical arrows are induced by $\phi$. Surjectivity of the second vertical arrow then follows from chasing this diagram. (This is one of the two ``four-lemmas'' that imply the ``five lemma''.) Likewise, the injectivity claim for $j$ follows by chasing the commutative diagram with exact rows
\[
\begin{CD}
H_k(\mc{G}_jC_*) @>>> H_{k-1}(\mc{F}_{j-1}C_*) @>>> H_{k-1}(\mc{F}_jC_*) @>>> H_{k-1}(\mc{G}_jC_*) \\
 @VV{\op{surj}}V @VV{\op{inj}}V @VVV @VV{\op{inj}}V\\
 H_k(\mc{G}_jC_*') @>>> H_{k-1}(\mc{F}_{j-1}C_*') @>>> H_{k-1}(\mc{F}_jC_*') @>>> H_{k-1}(\mc{G}_jC_*').
\end{CD}
\]
\end{proof}

Since \eqref{eqn:surjection} is a surjection, by taking the direct limit over $N$ and $H'$, and using an action-filtered version of Proposition~\ref{shortS1+}, we conclude that the map \eqref{eqn:ssis} is surjective as desired.

%% file: eh-transfer.tex
\section{Definition of transfer morphisms}
\label{section:transfer}

Let $(V,\lambda_V)$ and $(W,\lambda_W)$ be Liouville domains.
Let $\varphi:V\to W$ be a Liouville embedding, i.e.\ a smooth embedding such that $\varphi^{\star}\lambda_W=\lambda_V$. Assume also as in \S\ref{sec:shproperties} that $\varphi(V)\subset\op{int}(W)$. In this situation one can define a ``transfer morphism''
\begin{equation}
\label{eqn:transfermorphism}
\phi_{V,W}^{(S^1,+)} : SH^{(S^1,+)}(W,\lambda_W)\longrightarrow SH^{(S^1,+)}(V,\lambda_V).
\end{equation}
Here the superscript `$(S^1,+)$' means that the superscripts `$S^1$' and `$+$' are optional (but the same in all three places).

A transfer morphism for symplectic homology was defined by Viterbo \cite{V}, and extended by the first author in his PhD thesis \cite{gutt} for (positive) equivariant symplectic homology.  We now review what we need to know about the definition of the transfer morphisms \eqref{eqn:transfermorphism}, and then explain how to extend the construction to generalized Liouville embeddings as in Definition~\ref{def:gle}.

\subsection{Transfer morphisms for (positive) symplectic homology}
\label{sec:nonequivtransfer}

To construct transfer morphisms, we introduce a special class $\Hs_{stair}(V,W)$ of Hamiltonians on $S^1\times\widehat{W}$ called ``admissible stair Hamiltonians''.  The transfer morphism is defined as a direct limit of continuation morphisms between an admissible Hamiltonian $H_1 \in \Hstd(W)$ and an admissible stair Hamiltonian $H_2\in\Hs_{stair}(V,W)$.

Below, identify $V$ with its image under the Liouville embedding $\varphi$. Given $\delta>0$ small, there is a unique neighbourhood $U$ of $\partial V$ in $W\setminus \op{int}(V)$, together with a symplectomorphism
\[
(U,\omega_W) \simeq \bigl([0,\delta]\times \partial V, d(e^{\rho}\lambda_V)\bigr),
\]
such that the Liouville vector field for $\lambda_W$ on the left hand side corresponds to $\partial_\rho$ on the right hand side. Here $\rho$ denotes the $[0,\delta]$ coordinate.

\begin{definition}
\label{def:Hstair}
A Hamiltonian $H_{2} : S^1\times \widehat{W}\rightarrow\R$ is in $\Hs_{stair}(V,W)$ if and only if
\begin{description}
\item{(1)}
The restriction of $H_2$ to $S^1\times V$ is negative, autonomous (i.e.\ $S^1$-independent), and $C^2$-small (so that there are no non-constant 1-periodic orbits). Furthermore,
\begin{equation}
H > -\epsilon
\end{equation}
on $S^1\times V$, where $\epsilon=\frac{1}{2}\min\big\{\Spec(\partial V,\lambda_V) \cup \Spec(\partial W,\lambda_W)\big\}$.
\item{(2)}
On $S^1\times U\cong S^1\times[0,\delta]\times \partial V$, with $\rho$ denoting the $[0,\delta]$ coordinate, we have:
\begin{itemize}
\item
There exists $0<\rho_{0}<\frac{\delta}{4}$ such that for $\rho_{0}\leq\rho\leq \delta-\rho_{0}$ we have
\begin{equation}
\label{eqn:H2linear}
H_2(\theta, \rho, y) = \beta e^{\rho} + \beta',
\end{equation}
where $0<\beta\notin\Spec(\partial V,\lambda_V)\cup\Spec(\partial W,\lambda_W)$ and $\beta'\in\R$.
\item
There exists a strictly convex increasing function $h_1:[1,e^{\rho_0}]\to\R$ such that on $S^1\times[0,\rho_0]\times Y$, the function $H_2$ is $C^{2}$-close to the function sending $(\theta,\rho,p)\mapsto h_1(e^\rho)$. Here and in the rest of this definition, the meanings of ``close'' and ``small'' are as in Remarks~\ref{rem:orbitsofHstand} and \ref{rem:hsmall}.
\item
There exists a small, strictly concave, increasing function $h_2:[e^{\delta-\rho_0}, e^{\delta}]\to\R$ such that $h_2(e^\delta)-h_2$ is small, and on $S^1\times[\delta-\rho_0, \delta]\times Y$, the function $H_2$ is $C^{2}$-close to the function sending $(\theta,\rho,y)\mapsto h_2(e^\rho)$.
\end{itemize}
\item{(3)}
On $S^{1}\times W\setminus(V\cup U)$, the function $H_2$ is $C^{2}$-close to a constant.
\item{(4)}
On $S^1\times[0,+\infty)\times  \partial W$, with $\rho'$ denoting the $[0,\infty)$ coordinate, we have:
\begin{itemize}
\item There exists $\rho'_{1} > 0$ such that for $\rho'\ge\rho'_1$ we have
\[
H_2(\theta, \rho', p) = \mu e^{\rho'} + \mu',
\]
with $0<\mu\notin\Spec(\partial V,\lambda_V)\cup\Spec(\partial W,\lambda_W)$, $\mu<\frac{\beta(e^{\delta}-1)}{e^\delta}$, and $\mu'\in\R$.
\item
There exists a strictly convex, increasing function $h_3:[1,e^{\rho'_1}]\to\R$ such that $h_3-h_3(1)$ is small, and on $S^1\times[0,\rho'_1]\times Y$, the function $H_2$ is $C^{2}$-close to the function sending $(\theta,\rho',y)\mapsto h_3(e^{\rho'})$.
\end{itemize}
\item{(5)}
The Hamiltonian $H_2$ is nondegenerate, i.e.\ all $1$-periodic orbits of $X_{H_2}$ are nondegenerate.
\end{description}
We denote the set of admissible stair Hamiltonians by $\Hs_{stair}(V,W)$.
\end{definition}

The graph of an admissible stair Hamiltonian $H_2$ is shown schematically in Figure \ref{hamiltoniens}.
\begin{figure}
	\hspace{3cm}
	\def\svgwidth{1\textwidth}
	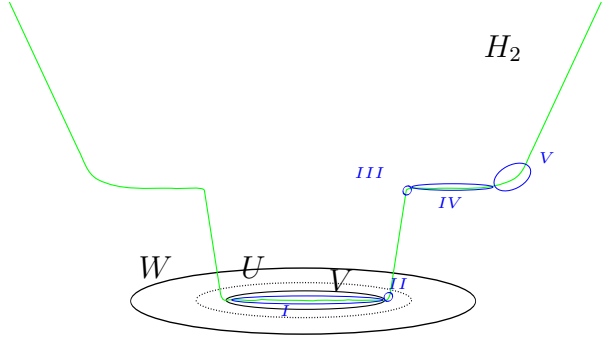
	\caption{\cite{gutt} Graph of an admissible stair Hamiltonian $H_2$ on $S^1\times \widehat{W}$}\label{hamiltoniens}
\end{figure}

The $1$-periodic orbits of $H_2$ lie either in the interior of $V$ (which we call region \rom{1}), in $[0,\rho_0]\times\partial V$ (region \rom{2}),
in $[\delta-\rho_0,\delta]\times\partial V$ (region \rom{3}), in $W\setminus(V\cup U)$ (region \rom{4}), or in $[0,\rho'_1]\times\partial W$ (region \rom{5}).
\begin{itemize}
\item[\textbf{\rom{1}}]
The $1$-periodic orbits in region \rom{1} correspond to critical points of $H_2$ on $V$.
\item[\textbf{\rom{2}}]
In region \rom{2}, the 1-periodic orbits are associated to Reeb orbits of $\lambda_V$ on $\partial V$ as in Remark~\ref{rem:hsmall}.
\item[\textbf{\rom{3}}]
In region \rom{3}, the 1-periodic orbits are likewise associated to Reeb orbits of $\lambda_V$ on $\partial V$.
\item[\textbf{\rom{4}}]
The $1$-periodic orbits in region \rom{4} correspond to critical points $H_2$ on $W\setminus(V\cup U)$.
\item[\textbf{\rom{5}}]
In region \rom{5}, the 1-periodic orbits are associated to Reeb orbits of $\lambda_W$ on $\partial W$.
\end{itemize}
The Hamiltonian actions of the $1$-periodic are ordered as follows:
\[	\mc{A}(\rom{4})<\mc{A}(\rom{5})<0<\mc{A}(\rom{1})<\mc{A}(\rom{2}).
\]
This means that every $1$-periodic orbit in region \rom{4} has Hamiltonian action less than every $1$-periodic orbit in region \rom{5}, and so forth. 

We now consider the Floer chain complex $CF(H_2,J_2)$ where $J_2:S^1\to\op{End}(T\widehat{W})$ is an $S^1$-family of almost complex structures on $\widehat{W}$. As in Definition~\ref{def:admJ}, we assume that $J_2^\theta$ is $\widehat{\omega}_W$-compatible for each $\theta\in S^1$, and that
\[
J_2^\theta(\partial_{\rho'}) = R_{\lambda_W}
\]
on $[\rho_1',\infty)\times\partial W$. This is enough to give a well-defined chain complex $CF(H_2,J_2)$, cf.\ \cite[\S1.2.3]{oanceathese}. We also assume that
\begin{equation}
\label{eqn:VReeb}
J_2^\theta(\partial_\rho) = R_{\lambda_V}
\end{equation}
on $[\rho_0,\delta-\rho_0]\times\partial V$.

Let $C^{\rom{1}, \rom{3}, \rom{4}, \rom{5}}(H_2,J_2)$ denote the subcomplex of $CF(H_2,J_2)$ generated by $1$-periodic orbits lying in regions \rom{1}, \rom{3}, \rom{4}, and \rom{5}. Let $C^{\rom{3}, \rom{4}, \rom{5}}(H_2,J_2)$  denote the subcomplex of $CF(H_2,J_2)$ generated by $1$-periodic orbits lying in regions \rom{3}, \rom{4} and  \rom{5}. These are subcomplexes because the action decreases along Floer trajectories, and \cite[Lem.\ 2.3]{CO} shows that there does not exist any Floer trajectory from region $\rom{3}$ to region $\rom{1}$ or $\rom{2}$. We then have quotient chain complexes
\[
\begin{split}
	C^{\rom{1},\rom{2}}(H_2,J_2) &= C^{\rom{1}, \rom{2}, \rom{3}, \rom{4}, \rom{5}}(H_2,J_2)/\raisebox{-1ex}{$C^{\rom{3}, \rom{4}, \rom{5}}(H_2,J_2)$}\\
	C^{\rom{2}}(H_2,J_2) &= C^{\rom{1}, \rom{2}, \rom{3}, \rom{4}, \rom{5}}(H_2,J_2)/\raisebox{-1ex}{$C^{\rom{1}, \rom{3}, \rom{4}, \rom{5}}(H_2,J_2)$}.
\end{split}
\]

Given $H_2$ and $J_2$ as above, let $H_2^V\in\Hstd(V)$ denote the admissible Hamiltonian for $V$ which agrees with $H_2$ on $V\cup ([0,\delta-\rho_0]\times\partial V)$, and which agrees with the right hand side of \eqref{eqn:H2linear} on $[\rho_0,\infty)\times \partial V$. Let $J_2^V$ denote the admissible $S^1$-family of almost complex structures on $\widehat{V}$ which agrees with $J_2$ on $V\cup ([0,\delta-\rho_0]\times\partial V)$, and which satisfies \eqref{eqn:VReeb} on $[\rho_0,\infty)\times \partial V$. Observe that we have canonical identifications of chain modules
\begin{equation}
\label{eqn:cicm}
\begin{split}
C^{\rom{1},\rom{2}}(H_2,J_2) &= CF\left(H_2^V,J_2^V\right),\\
C^{\rom{2}}(H_2,J_2) &= CF^{+}\left(H_2^V,J_2^V\right),
\end{split}
\end{equation}
because the generators on both sides correspond to the same $1$-periodic orbits in $V\cup([0,\delta-\rho_0]\times\partial V)$.

\begin{proposition}
\label{prop:homCII}
\cite[Proposition 4.4]{gutt}
The canonical identifications \eqref{eqn:cicm} induce isomorphisms on homology
\[
\begin{split}
H\bigl(C^{\rom{1},\rom{2}}(H_2,J_2),\partial\bigr) &= HF\left(H_2^V,J_2^V\right),\\
H\bigl(C^{\rom{2}}(H_2,J_2),\partial\bigr) &= HF^{+}\left(H_2^V,J_2^V\right).
\end{split}
\]
\end{proposition}

Given $H_2$ and $J_2$ as above, suppose that $H_1\in\Hstd(W)$ satisfies $H_1\le H_2$ pointwise. Let $J_1$ be an admissible $S^1$-family of almost complex structures on $\widehat{W}$. We then have a well-defined continuation map
\begin{equation}
\label{eqn:transcont}
HF(H_1,J_1) \longrightarrow HF(H_2,J_2)
\end{equation}
defined as in \eqref{eq:floerparam}.

\begin{definition}\label{def:transfermorphism}
We define the transfer morphism on Floer homology to be the composition
\[
\phi_{H_2^V,H_1}: HF(H_1,J_1)\longrightarrow HF(H_2,J_2)\longrightarrow H\bigl(C^{\rom{1},\rom{2}}(H_2,J_2)\bigr)
= HF(H_2^V,J_2^V).
\]
Here the first arrow is the continuation map \eqref{eqn:transcont}, the second map is induced by projection onto the quotient chain complex, and the equality sign on the right is the canonical isomorphism from Proposition~\ref{prop:homCII}. Concretely, this map counts solutions of equation \eqref{eq:floerparam} going from a $1$-periodic orbit of $X_{H_1}$ to a $1$-periodic orbit of $X_{H_2}$ lying in region $\rom{1}$ or $\rom{2}$.

Since the continuation map decreases action, it follows that in the above composition, we can start with the homology of the quotient by $CF^{\le \epsilon}(H_1,J_1)$, to obtain a transfer map on positive Floer homology,
\[
\phi_{H_2^V,H_1}^+:
HF^+(H_1,J_1)\longrightarrow H\left(\frac{CF(H_2,J_2)}{CF^{\le\epsilon}(H_2,J_2)}\right) \longrightarrow H\bigl(C^{\rom{2}}(H_2,J_2)\bigr)
= HF^+(H_2^V,J_2^V).
\]
\end{definition}

The above transfer maps $\phi_{H_2^V,H_1}$ and $\phi^+_{H_2^V,H_1}$ depend only on $H_1$ and $H_2^V$, and more generally commute with continuation maps for increasing $H_1$ and $H_2^V$; see \cite[Prop.\ 4.7]{gutt}. Consequently, we can define a transfer morphism on (positive) symplectic homology by taking direct limits:
\[
\begin{split}
\phi_{V,W} &= \lim_{\substack{\longrightarrow\\H_1,H_2^V}}\phi_{H_2^V,H_1}: SH(W,\lambda_W) \longrightarrow SH(V,\lambda_V),\\
\phi_{V,W}^+ &=  \lim_{\substack{\longrightarrow\\H_1,H_2^V}}\phi_{H_2^V,H_1}^+ : SH^+(W,\lambda_W) \longrightarrow SH^+(V,\lambda_V).
\end{split}
\]

\subsection{Transfer morphisms for (positive) $S^1$-equivariant symplectic homology}
\label{sec:equivtransfer}

Recall that to define (positive) $S^1$-equivariant symplectic homology, we modify the definition of (positive) symplectic homology, by replacing the notion of admissible Hamiltonians $H:S^1\times\widehat{X}\to\R$ in Definition~\ref{def:Hstd} by the notion of admissibile parametrized Hamiltonians $H:S^1\times\widehat{X}\times S^{2N+1}$ in Definition~\ref{def:aph}. In an analogous way, one can modify the definition of admissible stair Hamiltonians $H_2 : S^1\times \widehat{W}\rightarrow\R$ in Definition~\ref{def:Hstair}, to define a notion of ``admissible parametrized stair Hamiltonians'' $H_2:S^1\times\widehat{W}\times S^{2N+1}\to\R$. We can then repeat the constructions in \S\ref{sec:nonequivtransfer} to obtain transfer maps
\begin{align}
\nonumber
\phi_{H_2^V,H_1}^{S^1}: HF^{S^1,N}(H_1) &\longrightarrow HF^{S^1,N}(H_2^V),\\
\label{eqn:phi+}
\phi_{H_2^V,H_1}^{S^1,+}: HF^{S^1,N,+}(H_1) &\longrightarrow HF^{S^1,N,+}(H_2^V).
\end{align}
We can then take the direct limit over $H_1$, $H_2^V$, and $N$ to define transfer morphisms
\[
\begin{split}
\phi_{V,W}^{S^1}: SH^{S^1}(W,\lambda_W) &\longrightarrow SH^{S^1}(V,\lambda_V),\\
\phi_{V,W}^{S^1,+}: SH^{S^1,+}(W,\lambda_W) &\longrightarrow SH^{S^1,+}(V,\lambda_V).
\end{split}
\]

\begin{remark}
\label{rem:simplifyingtransfer}
One can also describe the transfer morphism \eqref{eqn:phi+} for fixed $N$ in the context of Remark~\ref{rmk:simplifyingcomplex} and Proposition~\ref{shortS1+}. Here one starts with an admissible stair Hamiltonian $H_2':S^1\times\widehat{W}\to\R$ and an admissible Hamiltonian $H_1':S^1\times\widehat{X}\to\R$ with $H_1'\le H_2'$. Recall that the homology $HF^{S^1,N,+}(H_1')$ appearing in Proposition~\ref{shortS1+} is the homology of a chain complex generated by symbols $u^k\tensor \gamma$, where $k\in\{0,\ldots,N\}$ and $\gamma$ is a nonconstant $1$-periodic orbit of $X_{H_1'}$. The differential has the form
\[
{\partial}_1^{S^1}(u^k\tensor\gamma) = \sum_{i=0}^ku^{k-i}\tensor\varphi_{1,i}(\gamma).
\]
Likewise, the homology $HF^{S^1,N,+}((H_2')^V)$ is the homology of a chain complex generated by symbols $u^k\tensor \gamma$, where $k\in\{0,\ldots,N\}$ and $\gamma$ is a nonconstant $1$-periodic orbit of $X_{(H_2')^V}$. The differential has the form
\[
{\partial}_2^{S^1}(u^k\tensor\gamma) = \sum_{i=0}^ku^{k-i}\tensor\varphi_{2,i}(\gamma).
\]
We now construct the transfer map \eqref{eqn:phi+} using continuation maps for homotopies which respect the inclusions $\widetilde{\imath}_0$ and $\widetilde{\imath}_1$ as in Remark~\ref{rmk:simplifyingcomplex}. This transfer map will then be induced by a chain map
having the form
\begin{equation}
\label{eqn:transferpsi}
\psi(u^k\tensor\gamma) = \sum_{i=0}^ku^{k-i}\tensor\psi_i(\gamma).
\end{equation}
\end{remark}

\subsection{Transfer morphisms for generalized Liouville embeddings}

We now extend the definition of transfer morphisms for a generalized Liouville embedding $\varphi:(V,\lambda_V) \to (W,\lambda_W)$ with $\varphi(V)\subset\op{int}(W)$. 

\begin{lemma}
\label{lem:extensionlambda}
Let $\varphi:(V,\lambda_V)\hookrightarrow(W,\lambda_W)$ be a generalized Liouville embedding with $\varphi(V)\subset\op{int}(W)$. Then there exists a 1-form $\lambda_W'$ on $W$ such that
\begin{enumerate}
\item $d\lambda_W'=d\lambda_W$,
\item $\lambda_W'=\lambda_W$ near $\partial W$,
\item $\varphi^{\star}\lambda_W'=\lambda_V$.
\end{enumerate}
\end{lemma}

\begin{proof}
Given $\delta>0$, define
\[
V_\delta = V\cup ([0,\delta]\times\partial V) \subset \widehat{V}.
\]
As in \cite[Thm.\ 3.3.1]{ms3}, if $\delta$ is sufficiently small then we can extend $\varphi$ to a symplectic embedding
\[
\varphi_\delta: (V_\delta,\widehat{\omega_V}) \longrightarrow (W,\omega_W).
\]
Now use the map $\varphi_\delta$ to identify $V_\delta$ with its image in $W$.  Then the $1$-form $\lambda_W - \widehat{\lambda_V}$ is closed on $V_\delta$.

By hypothesis, the de Rham cohomology class of this $1$-form restricted to $[0,\delta]\times \partial V$ is zero. Thus there is a function $g:[0,\delta]\times\partial V$ such that
\[
dg = (\lambda_W - \widehat{\lambda_V})\big|_{[0,\delta]\times\partial V}.
\]
Let $\beta:[0,\delta]\to\R$ be a smooth function with $\beta(\rho)\equiv 0$ for $\rho$ close to $0$ and $\beta(\rho)\equiv 1$ for $\rho$ close to $\delta$. We can then take
\[
\lambda_W' = \left\{\begin{array}{cl} \lambda_V & \mbox{on $V$},\\
\widehat{\lambda_V}+d(\beta g) & \mbox{on $[0,\delta]\times \partial V$},\\
\lambda_W & \mbox{on $W\setminus V_\delta$}.
\end{array}
\right.
\]
\end{proof}

Now given a generalized Liouville embedding as above, let $\lambda_W'$ be a $1$-form on $W$ provided by Lemma~\ref{lem:extensionlambda}. We then have an honest Liouville embedding
\[
\varphi:(V,\lambda_V) \longrightarrow (W,\lambda_W').
\]
As explained in \S\ref{sec:nonequivtransfer} and \S\ref{sec:equivtransfer}, this induces transfer maps
\begin{equation}
\label{eqn:transfer'}
SH^{(S^1,+)}(W,\lambda_W') \longrightarrow SH^{(S^1,+)}(V,\lambda_V).
\end{equation}

The construction in \S\ref{sec:equivariantsh} of (positive, $S^1$-equivariant) symplectic homology of $(W,\lambda_W)$ depends only on the contact form $\lambda_W|_{\partial W}$ on the boundary, and the symplectic form $\omega_W=d\lambda_W$ on the interior. Indeed, replacing the Liouville form $\lambda_W$ by another Liouville form $\lambda_W'$ with the same exterior derivative and restriction to the boundary does not change any of the chain complexes or maps in the definition of (positive, $S^1$-equivariant) symplectic homology\footnote{One might worry that the Hamiltonian action of a noncontractible loop can change if $\lambda_W-\lambda_W'$ is not exact. However for the Hamiltonians that we are using, the only noncontractible $1$-periodic orbits are associated to Reeb orbits and their action does not change.}, since the classes of admissible Hamiltonians used are determined by the restriction to the boundary, and the Hamiltonian vector fields are determined by the symplectic form. (For stronger results on invariance of symplectic homology see \cite[\S4.3]{gutt}.) Thus we have a canonical isomorphism
\begin{equation}
\label{eqn:shci}
SH^{(S^1,+)}(W,\lambda_W) = SH^{(S^1,+)}(W,\lambda_W').
\end{equation}
We can now finally make the following definition:

\begin{definition}
Suppose $\varphi:(V,\lambda_V)\to (W,\lambda_W)$ is a generalized Liouville embedding with $\varphi(V)\subset\op{int}(W)$. Let $\lambda_W'$ be a $1$-form provided by Lemma~\ref{lem:extensionlambda}. Define the {\bf transfer morphism\/}
\begin{equation}
\label{eqn:gletransfer}
\phi_{V,W}^{(S^1,+)}: SH^{(S^1,+)}(W,\lambda_W) \longrightarrow SH^{(S^1,+)}(V,\lambda_V)
\end{equation}
to be the composition of the canonical isomorphism \eqref{eqn:shci} with the map \eqref{eqn:transfer'}.
\end{definition}

The transfer morphism \eqref{eqn:gletransfer} does not depend on the choice of $\lambda_W'$, because the admissible Hamiltonians, chain complexes, and chain maps in the definition of the transfer morphism depend only on the symplectic form on each Liouville domain and the contact form on the boundary of each Liouville domain.

%% file: hamiltoniens.pdf_tex
\begingroup%
  \makeatletter%
  \providecommand\color[2][]{%
    \errmessage{(Inkscape) Color is used for the text in Inkscape, but the package 'color.sty' is not loaded}%
    \renewcommand\color[2][]{}%
  }%
  \providecommand\transparent[1]{%
    \errmessage{(Inkscape) Transparency is used (non-zero) for the text in Inkscape, but the package 'transparent.sty' is not loaded}%
    \renewcommand\transparent[1]{}%
  }%
  \providecommand\rotatebox[2]{#2}%
  \ifx\svgwidth\undefined%
    \setlength{\unitlength}{836.840679bp}%
    \ifx\svgscale\undefined%
      \relax%
    \else%
      \setlength{\unitlength}{\unitlength * \real{\svgscale}}%
    \fi%
  \else%
    \setlength{\unitlength}{\svgwidth}%
  \fi%
  \global\let\svgwidth\undefined%
  \global\let\svgscale\undefined%
  \makeatother%
  \begin{picture}(1,0.31565916)%
    \put(0,0){\includegraphics[width=\unitlength]{hamiltoniens.pdf}}%
    \put(0.42156304,0.28630916){\color[rgb]{0,0,0}\makebox(0,0)[lt]{\begin{minipage}{0.19819025\unitlength}\raggedright $H_2$\end{minipage}}}%
    \put(0.11533968,0.08967333){\color[rgb]{0,0,0}\makebox(0,0)[lt]{\begin{minipage}{0.12357743\unitlength}\raggedright $W$\end{minipage}}}%
    \put(0.20627403,0.08967333){\color[rgb]{0,0,0}\makebox(0,0)[lt]{\begin{minipage}{0.11269642\unitlength}\raggedright $U$\end{minipage}}}%
    \put(0.28399569,0.07801508){\color[rgb]{0,0,0}\makebox(0,0)[lt]{\begin{minipage}{0.13134962\unitlength}\raggedright $V$\end{minipage}}}%
    \put(0.24104539,0.0460693){\color[rgb]{0,0,0}\makebox(0,0)[lt]{\begin{minipage}{0.54856967\unitlength}\raggedright \tiny{\color{blue}$\rom{1}$}\end{minipage}}}%
    \put(0.33664304,0.07019205){\color[rgb]{0,0,0}\makebox(0,0)[lt]{\begin{minipage}{0.54052873\unitlength}\raggedright \tiny{\color{blue}$\rom{2}$}\end{minipage}}}%
    \put(0.30715966,0.16846999){\color[rgb]{0,0,0}\makebox(0,0)[lt]{\begin{minipage}{0.54678281\unitlength}\raggedright \tiny{\color{blue}$\rom{3}$}\end{minipage}}}%
    \put(0.38042141,0.14256038){\color[rgb]{0,0,0}\makebox(0,0)[lt]{\begin{minipage}{0.52623375\unitlength}\raggedright \tiny{\color{blue}$\rom{4}$}\end{minipage}}}%
    \put(0.47065843,0.18187155){\color[rgb]{0,0,0}\makebox(0,0)[lt]{\begin{minipage}{0.54410249\unitlength}\raggedright \tiny{\color{blue}$\rom{5}$}\end{minipage}}}%
  \end{picture}%
\endgroup%

%% file: eh-proofs2.tex
\section{Properties of transfer morphisms}
\label{sec:proofs2}

Let $\varphi:(X,\lambda) \to (X',\lambda')$ be a generalized Liouville embedding with $\varphi(X) \subset \op{int}(X')$. Let
\[
\Phi: CH(X',\lambda') \longrightarrow CH(X,\lambda)
\]
denote the transfer map $\phi_{X,X'}^{S^1,+}$ defined in \S\ref{section:transfer}. We now prove that this map satisfies the properties in Proposition~\ref{prop:transfer}.

\subsection{Action}

The transfer map $\Phi$ is a direct limit over $H_1$, $H_2^X$, and $N$ of continuation maps
\begin{equation}
\label{eqn:tcm1}
HF^{S^1,N,+}(H_1) \longrightarrow HF^{S^1,N,+}(H_2^X)
\end{equation}
where $H_1$ and $H_2^X$ are appropriate parametrized Hamiltonians for $X'$ and $X$ respectively. Since the continuation map \eqref{eqn:tcm1} is induced by a chain map which decreases symplectic action, it is the direct limit over $L$ of maps
\begin{equation}
\label{eqn:tcm2}
HF^{S^1,N,+,\le L}(H_1) \longrightarrow HF^{S^1,N,+,\le L}(H_2^X).
\end{equation}
We now define
\[
\Phi^L: CH^L(X',\lambda') \longrightarrow CH(X,\lambda)
\]
to be the direct limit over $H_1$, $H_2^X$, and $N$ of the maps \eqref{eqn:tcm2}. Here, as in \S\ref{pf:actionfiltration}, we assume without loss of generality that $L\notin\Spec(\partial X',\lambda') \cup \Spec(\partial X,\lambda)$. The required properties \eqref{eqn:filtran} and \eqref{eqn:iLP} follow from Definition~\ref{def:il2l1}.

\subsection{Commutativity with $U$}

We now show that the transfer map $\Phi$ commutes with the $U$ map defined in \S\ref{sec:U}. 

Recall that the map $\Phi$ can be computed as a direct limit of maps \eqref{eqn:transferpsi} from Remark~\ref{rem:simplifyingtransfer}. And recall from \S\ref{sec:U} that in this setup, the $U$ map is the direct limit of chain maps given by ``multiplication by $u^{-1}$''. So it is enough to prove that for each nonnegative integer $N$, we have a commutative diagram of chain maps
\[
\begin{CD}
{CF}^{S^1,N,+}(H_1') @>{\psi}>> {CF}^{S^1,N,+}\big((H_2')^V\big)\\
@V{u^{-1}}VV @VV{u^{-1}}V\\
{CF}^{S^1,N,+}(H_1') @>{\psi}>> {CF}^{S^1,N,+}\big((H_2')^V\big).
\end{CD}
\]
Here the chain complexes depend on $S^{2N+1}$-families of Hamiltonians and almost complex structures as in Remark~\ref{rmk:simplifyingcomplex}, which we are omitting from the notation.

It is enough to check this commutativity on a generator $u^{k}\tensor\gamma$. If $k=0$, then both compositions are zero, since $\psi$ does not increase the exponent of $k$. If $k>0$, then the lower left composition is given by
\[
\psi \big(u^{-1} (u^k\tensor\gamma)\big) = \psi(u^{k-1}\tensor\gamma) = \sum_{i=0}^{k-1} u^{k-1-i}\tensor \psi_i(\gamma),
\]
while the upper right composition is given by
\[
u^{-1}\psi(u^k\tensor\gamma) = u^{-1}\sum_{i=0}^ku^{k-i}\tensor\psi_i(\gamma) = \sum_{i=0}^{k-1}u^{k-i-1}\tensor\psi_i(\gamma).
\]
These are equal, and this completes the proof that $\Phi U = U \Phi$.

To prove that $\Phi^L U^L = U^L\Phi^L$, as before we can assume without loss of generality that $L\notin\Spec(\partial X',\lambda')\cup\Spec(\partial X,\lambda)$. We then repeat the above argument, restricted to orbits with action less than $L$.

\subsection{Commutativity with $\delta$}

To conclude, we now prove the commutativity with $\delta$ in Proposition~\ref{prop:transfer}. Note that a closely related result was proved in \cite[Thm.\ 5.2]{V}, and our proof will use some of the same ideas.

Recall that the $\delta$ map is defined starting from the short exact sequence of chain complexes \eqref{eqn:sescc}. If $H_1$ and $H_2^X$ are Hamiltonians as in the definition of the transfer map in \S\ref{sec:equivtransfer}, then we have a commutative diagram
\[
\begin{CD}
C_{\op{Morse}}(X',H_1)\tensor\Q\{1,u,\ldots,u^N\} @>>> CF^{S^1,N}(H_1,J_1) @>>> CF^{S^1,N,+}(H_1,J_1)\\
 @VVV @VVV @VVV \\
C_{\op{Morse}}(X,H_2^X)\tensor\Q\{1,u,\ldots,u^N\} @>>> CF^{S^1,N}(H_2^X,J_2^X) @>>> CF^{S^1,N,+}(H_2^X,J_2^X).
\end{CD}
\]
Here the rows are from the short exact sequences of chain complexes \eqref{eqn:sescc} for $X'$ and $X$. The center vertical arrow is the continuation chain map which, in the direct limit, gives the transfer morphism $\phi_{X,X'}^{S^1}$. The right vertical arrow is the continuation chain map which, in the direct limit, gives the transfer morphism $\Phi=\phi_{X,X'}^{S^1,+}$. The left vertical arrow is the restriction of the center vertical arrow. As in the proof of \cite[Thm.\ 5.2]{V}, this left arrow simply discards critical points in $X'\setminus X$ (here we are identifying $X$ with its image in $X'$ under the symplectic embedding), and is the Morse continuation map from ${H_1}|_X$ to ${H_2}|_X$.

The above commutative diagram gives rise to a morphism of long exact sequences on homology. One square of this is the commutative diagram
\[
\begin{CD}
HF^{S^1,N,+}(H_1,J_1) @>>> H_*(X',\partial X') \tensor \Q\{1,u,\ldots,u^N\}\\
@V{\phi^{S^1,+}_{H_2^X,H_1}}VV @VV{\rho\tensor 1}V\\
HF^{S^1,N,+}(H_2^X,J_2^X) @>>> H_*(X,\partial X)\tensor\Q\{1,u,\ldots,u^N\}. 
\end{CD}
\] 
Here the horizontal arrows are the connecting homomorphisms which, in the direct limit, give the $\delta$ maps for $X'$ and $X$. Thus taking the direct limit over $N$, $H_1$, and $H_2^X$, we obtain the desired commutative diagram \eqref{eqn:commdelta}.